\RequirePackage{fix-cm}
\documentclass[smallextended]{svjour3}       
\smartqed  
\usepackage{braket,amsfonts}

\usepackage{array}

\usepackage[caption=false]{subfig}

\usepackage{algorithmic}

\usepackage{graphicx,epstopdf}

\usepackage{amsopn}

\usepackage{xspace}
\usepackage{bold-extra}
\usepackage[most]{tcolorbox}

\colorlet{texcscolor}{blue!50!black}
\colorlet{texemcolor}{red!70!black}
\colorlet{texpreamble}{red!70!black}
\colorlet{codebackground}{black!25!white!25}

\lstdefinestyle{siamlatex}{%
  style=tcblatex,
  texcsstyle=*\color{texcscolor},
  texcsstyle=[2]\color{texemcolor},
  keywordstyle=[2]\color{texemcolor},
  moretexcs={cref,Cref,maketitle,mathcal,text,headers,email,url},
}

\tcbset{%
  colframe=black!75!white!75,
  coltitle=white,
  colback=codebackground, 
  colbacklower=white, 
  fonttitle=\bfseries,
  arc=0pt,outer arc=0pt,
  top=1pt,bottom=1pt,left=1mm,right=1mm,middle=1mm,boxsep=1mm,
  leftrule=0.3mm,rightrule=0.3mm,toprule=0.3mm,bottomrule=0.3mm,
  listing options={style=siamlatex}
}

\newtcbinputlisting[use counter=example]{\examplefile}[3][]{%
  title={Example~\thetcbcounter: #2},listing file={#3},#1}

\DeclareTotalTCBox{\code}{ v O{} }
{ 
  fontupper=\ttfamily\color{black},
  nobeforeafter,
  tcbox raise base,
  colback=codebackground,colframe=white,
  top=0pt,bottom=0pt,left=0mm,right=0mm,
  leftrule=0pt,rightrule=0pt,toprule=0mm,bottomrule=0mm,
  boxsep=0.5mm,
  #2}{#1}

\usepackage{booktabs}
\usepackage{color}
\usepackage{amsmath}
\usepackage{amssymb}
\allowdisplaybreaks
\usepackage{algorithm}

\newtheorem{assumption}{Assumption}

\def\sS{\mathcal{S}}

\def\bR{\mathbb{R}}

\def\sI{\mathcal{I}}

\def\sL{\mathcal{L}}

\def\sF{\mathcal{F}}
\def\sO{\mathcal{O}}

\def\fl{f_{\rm low}}

\def\b1{{\bf 1}}

\def\rev#1{\textcolor{black}{#1}}
\def\revn#1{\textcolor{black}{#1}}
\def\Y{Y}

\def\sX{\mathcal{X}}
\def\sJ{\mathcal{J}}
\newcommand{\csol}{c_{\textnormal{sol}}}
\newcommand{\cnc}{c_{\textnormal{nc}}}
\newcommand{\bH}{\bar{H}}
\newcommand{\eps}{\epsilon}
\def\diag{\mathop{\hbox{\rm diag}}}
\def\be{\begin{enumerate}}
\def\bi{\begin{itemize}}
\def\ee{\end{enumerate}}
\def\ei{\end{itemize}}
\def\sgn{\mathop{\hbox{\rm sgn}}}
\newcommand{\Kpncg}{K_{\mbox{\rm\scriptsize pncg}}}
\newcommand{\Cmeo}{\mathcal{C}^{\mbox{\rm\scriptsize meo}}}
\newcommand{\Nmeo}{N^{\mbox{\rm\scriptsize meo}}}

\newcommand{\soda}{\text{\bf Def1}}
\newcommand{\sodb}{\text{\bf Def2}}
\newcommand{\sodc}{\text{\bf Def3}}
\newcommand{\sodd}{\text{\bf Def4}}
\newcommand{\sode}{\text{\bf Def5}}
\newcommand{\sodf}{\text{\bf Def7}}

\newcommand{\sodAA}{\text{\bf DefA}}
\newcommand{\sodBB}{\text{\bf DefB}}
\newcommand{\sodCC}{\text{\bf DefC}}

\def\sjw#1{\textcolor{blue}{\bf [SJW: #1]}}
\begin{document}

\title{Complexity of \rev{a Projected Newton-CG Method for Optimization with Bounds}
\thanks{A preliminary version of this work has been archived in the workshop ``Beyond First-Order Methods in ML Systems'' at the 37th International Conference on Machine Learning, Vienna, Austria, 2020. Research is supported from NSF Awards 1740707, 1839338, 1934612, and 2023239; Subcontract 8F-30039 from
Argonne National Laboratory; Award N660011824020 from the DARPA Lagrange Program; HKU-IDS start-up fund;and Guangdong Province Fundamental and Applied Fundamental Research Regional Joint Fund, 2022B1515130009. \rev{This work was submitted when the first author was a postdoctoral research associate at the Wisconsin Institute for Discovery at University of Wisconsin-Madison.}
}
}


\author{Yue Xie         \and
        Stephen J. Wright 
}


\institute{Yue Xie \at
              \rev{Department of Mathematics and Musketeers Foundation Institute of Data Science, The University of Hong Kong, Pokfulam, Hong Kong.}\\
              \email{yxie21@hku.hk}           
           \and
           Stephen J. Wright \at
              Computer Sciences Department, University of Wisconsin-Madison,  1210 W. Dayton St., Madison, WI, 53706. \\
              \email{swright@cs.wisc.edu} 
}

\date{Received: date / Accepted: date}

\maketitle

\begin{abstract}
\rev{This paper describes a method for solving smooth nonconvex minimization problems subject to bound constraints with good worst-case complexity guarantees and practical performance.
The method contains elements of two existing methods: the classical gradient projection approach for bound-constrained optimization and a recently proposed Newton-conjugate gradient algorithm for unconstrained nonconvex optimization.
Using a new definition of approximate second-order optimality parametrized by some tolerance $\epsilon$ (which is compared with related definitions from previous works), we derive complexity bounds in terms of $\epsilon$ for both the number of iterations required and the total amount of computation. The latter is measured by the number of gradient evaluations or Hessian-vector products. We also describe illustrative computational results on several test problems from low-rank matrix optimization.}
\end{abstract}

\section{\rev{Introduction}} \label{sec:intro}
We consider the problem
\begin{equation} \label{opt: bc0}
\min \, f(x)  \quad \mbox{subject to} \;\;  x \in \Omega,
\end{equation}
where $f: \bR^{n} \rightarrow \bR$ is twice continuously differentiable and is bounded below by $\fl>-\infty$ on the closed feasible set $\Omega$. We focus on $\Omega$ defined by nonnegativity constraints on a subset $\sI$ of the variables, that is,
\begin{equation} \label{def:O}
    \Omega \triangleq \{ x \in \bR^n \mid x^i \ge 0, \; i \in \sI \}, \;\; \mbox{where $\sI \subseteq \{1,2,\dotsc,n\}$.}
\end{equation}
Bounds are the simplest type of inequality constraint.
Euclidean projection onto the feasible set $\Omega$, a trivial operation when $\Omega$ is defined by bounds, is a fundamental component of several successful algorithms.
Bound-constrained subproblems often arise in algorithms for more complicated constrained optimization problems, such as augmented Lagrangian methods. 
\revn{Bound constraints also appear in popular problems such as nonnegative least-squares and nonnegative matrix factorization \cite{gillis2014and}.} 
\rev{Approaches of several types have been proposed for solving this problem, including 
gradient projection, active set methods, and interior-point methods. \revn{See \cite{nocedal2006numerical} for details.}}


In this paper, we describe a line-search method for solving \eqref{opt: bc0}, \eqref{def:O} that exploits the simplicity of Euclidean projection onto $\Omega$. It combines gradient projection with a Newton-conjugate gradient (Newton-CG) method  for smooth nonconvex unconstrained optimization proposed recently in \cite{Royer2019}. 
The elements of our method are well known for their good practical performance in various optimization contexts.
By combining these elements in the right way, and introducing judicious strategies for diagonal scaling, step length acceptance, and detection of negative curvature, we equip the method with a worst-case complexity theory that matches best-known theoretical bounds for bound-constrained optimization and even for unconstrained optimization. Preliminary numerical results confirm that the method has appealing practical performance.
In contrast to most previous works on complexity, we prove results for both iteration and {\em computational} complexity. The latter is measured in terms of two key operations: evaluation of a gradient at a given point, and computation of a Hessian-vector product involving an arbitrary vector. (The latter is known to cost a modest multiple of a gradient evaluation when computational differentiation techniques are used \cite{griewank2008evaluating}.) \revn{Our method does not require explicit calculation or storage of the Hessian; it accesses the Hessian only via products with given vectors.}

\paragraph{Background and Prior Work.}
\rev{There has been renewed interest in devising optimization algorithms with worst-case complexity guarantees for constrained nonconvex optimization. 
Interior-point type methods were developed to solve nonconvex problems with bound constraints \cite{Bian2015}, or with bounds and linear equality constraints \cite{Haeser2018}. A log-barrier method for bound-constrained problems was proposed in \cite{10.1093/imanum/drz074}.
Like the present paper, this method made use of the Newton-CG method of \cite{Royer2019}, but in a quite different way.
An adaptive cubic regularization algorithm was proposed in \cite{cartis2012adaptive} to solve nonconvex optimization with general convex constraints. 
Later, in  \cite{cartis2015evaluation}, the authors of  \cite{cartis2012adaptive} designed a novel two-phase target-following algorithm to address a more general problem class: nonconvex optimization with nonlinear equality constraints and a general convex feasible region.
They also generalize the concept of approximate first-order optimal point to arbitrary high-order and apply a conceptual high-order algorithm for obtaining such a point \cite{cartis2018second}. 
Authors of \cite{birgin2018regularization} outline a high-order algorithm that obtains approximate first-order optimal point of a nonconvex optimization problem with general constraints. }
\revn{The paper \cite{cartis2019universal} considered a high-order universal adaptive regularization algorithm to find approximate first-order optimal points for nonconvex problems with convex constraints, but they have even less stringent assumptions on the smoothness of the objective.
Specifically, they required $q$th-order derivatives to be H\"older continuous, and they obtained complexity results that depend on the degree of smoothness and / or the regularization power\footnote{Order of the regularization term. For example, a cubic regularization has power $3$.} (see details below). 
For high-order adaptive regularization methods,  \cite{chen2017partially} showed that the complexity may not be affected when a non-Lipschitz singular function ($l_p$-norm, $p \in (0,1)$) is introduced into the objective. 
Other works include \cite{nouiehed2020trust},  which uses a  trust region method  to locate second-order optimal point of nonconvex problems with linear constraints.} \rev{A Hessian barrier algorithm was recently proposed \cite{dvurechensky2021hessian}, based on self-concordant barrier functions, which solves nonconvex problems with general conic constraints and linear equality constraints. }

\rev{In these articles, good complexity results follow from the use of the Hessian and sometimes higher-order derivatives: $\sO(\epsilon^{-3/2})$ iteration/evaluation\footnote{\rev{Iteration complexity in this paper is a bound on the number of outer iterations in an algorithm. It is equivalent to evaluation complexity (a count of the number of evaluations of gradients, Hessians, or higher-order derivatives)  for purposes of this discussion.}} complexity to locate an $\epsilon$-approximate first-order optimal point \cite{cartis2012adaptive,cartis2015evaluation,birgin2018regularization} or a $(\epsilon,\sqrt{\epsilon})$ second-order optimal point \cite{Bian2015,Haeser2018,10.1093/imanum/drz074,dvurechensky2021hessian,nouiehed2020trust}. 
(Here $\epsilon$ and $\sqrt{\epsilon}$ represent the precision of first- and second-order optimality conditions, respectively.)
The $q$th-order algorithm in \cite{cartis2018second} locates an $\epsilon$-approximate $q$th-order solution in $\sO(\epsilon^{-(q+1)})$  iterations; while the $q$th-order algorithms in \cite{birgin2018regularization} finds an $\epsilon$-approximate first-order solution in $\sO(\epsilon^{-(q+1)/q})$ iterations}. \revn{The algorithm that exploits the $q$th-order Taylor model in \cite{cartis2019universal} locates the $\epsilon$-approximate first-order solution in $\sO(\epsilon^{-(q +\alpha)/(q + \alpha-1)})$ iterations under the assumption that the objective's $q$th-order derivative (with $q \ge 1$) is H\"older continuous with exponent $\alpha$ (with $0 < \alpha \le 1$) and the regularization power is high enough.}

\rev{Complexity results in the works discussed above focus on iteration/evaluation complexity; less attention is paid to the bounds on the total amount of computation required. In fact, these methods can require solution of nonconvex subproblems that may themselves require a significant and undetermined amount of computation. 
For example, in \cite{cartis2012adaptive,cartis2015evaluation}, a potentially \revn{expensive} cubic regularized subproblem (itself a constrained nonconvex problem) needs to be solved to approximate first-order optimality at each iteration, while the higher-order methods of \cite{cartis2018second}, \cite{birgin2018regularization}, \cite{cartis2019universal} and \cite{chen2017partially} require solution of subproblems involving higher-order derivatives. \revn{In \cite{nouiehed2020trust}, checking the second-order stationary condition can be NP-hard, and the constrained nonconvex subproblem needs to be solved to at least first-order stationary per iteration.}
Moreover, implementations of these methods may require explicit evaluation of the Hessian or higher-order derivatives.
The method of this paper, by contrast, requires explicit evaluation only of gradients; the Hessians are accessed only via Hessian-vector products. 
This fact allows us to define meaningful bounds on computational complexity.}

\rev{
The pursuit of optimal iteration/evaluation complexity results may compromise the practicality of algorithms. For example, subproblems in the second-order algorithms from \cite{Bian2015} and \cite{Haeser2018} have a small trust-region radius that depends on $\epsilon$. 
The log-barrier approach of \cite{10.1093/imanum/drz074} has unimpressive practical performance, as we see in Section~\ref{sec: num}.}

\rev{
Other works that address complexity of constrained nonconvex optimization, include \cite{curtis2018complexity}, which discusses the trust funnel algorithm to solve optimization with equality constraints;  \cite{xie2021complexity,grapiglia2021complexity,birgin2020complexity,sahin2019inexact}, which discuss augmented Lagrangian methods (ALM); and \cite{lin2022complexity}, concerning penalty methods.
In \cite{grapiglia2021complexity}, ALM and appropriate first-order algorithms to solve subproblems are utilized to locate $\epsilon$ approximate first-order point, with evaluation complexity arbitrarily close to $\sO(\epsilon^{-4})$. 
Complexity of a safeguarded ALM is derived in \cite{birgin2020complexity} to find first-order stationary points, but the cost of solving the subproblems  is not well defined. 
In \cite{lin2022complexity}, complexity results are established in terms of the number of proximal gradient steps needed to find an $\epsilon$ first-order stationary points. 
The complexity can be improved to $\sO(\epsilon^{-5/2})$ (omitting logarithm terms) when the constraint functions are convex and Slater's condition holds. \cite{curtis2018complexity,xie2021complexity,sahin2019inexact} consider optimization with equality constraints that do not accommodate the bound-constrained problem class \eqref{opt: bc0}, \eqref{def:O}. 
}

\rev{A complicating factor in comparing complexity of methods for finding approximate optimal points is that the definitions of such points vary between papers. This is not unexpected since different papers consider a variety of constraint types, and the approximate optimality conditions are adapted to the particular formulations. 
The relation between different definitions has not been discussed in any detail, even for the case of  optimization with bounds. 
We believe that a proper discussion facilitates a better understanding of the goals and characteristics of different algorithms.}


\paragraph{Approach and Contributions.}
\rev{We describe an algorithm for locating an approximate second-order point of the problem \eqref{opt: bc0},\eqref{def:O} that has good worst-case complexity bounds --- similar to the unconstrained case ($\Omega = \bR^n$ in \eqref{opt: bc0}) --- and is also practical.}

\rev{As a preliminary to our description of the algorithm, we state our definition of 
approximate second-order optimality, alongside four other definitions that have appeared in the literature.
These definitions are typically parametrized by a tolerance $\epsilon$.
We introduce a second parameter $p$ that represents the power of $\epsilon$ that determines the approximate condition involving the Hessian, and refer to the resulting conditions as ``$(\epsilon,p)$-second-order optimality" or  ``$(\epsilon,p)$-2o" for short. 
The alternative definitions that we discuss in this article are based on those from \cite{cartis2018second,Haeser2018,10.1093/imanum/drz074,Bian2015}, specialized to the bound-constrained problem \eqref{opt: bc0},\eqref{def:O}, with $\sI = \{1,\hdots,n \}$.
We make comparisons among all these definitions, using a new notion of ``essentially stronger''.}
\rev{Practical methods that make use of gradient projection and Newton scaling have yet to be considered seriously as methods with good complexity guarantees for bound-constrained problems.
Such methods exploit the simplicity of the projection operation for $\Omega$ in \eqref{def:O}, as well as the benefits of second-order information that have been shown in the unconstrained context. 
The two-metric projection framework proposed by Bertsekas \cite{bertsekas1982projected,bertsekas2014constrained} provides a potential framework, for appropriate choice of scaling matrix.}
This method takes steps of the form
\begin{equation} \label{eq:2mgp}
x_{k+1} \triangleq P(x_k - \alpha_k D_k \nabla f(x_k)),
\end{equation}
where $D_k$ is a symmetric positive definite matrix (with a certain structure defined below) and $P(z)$ is the projection onto the feasible set $\Omega$ in \eqref{def:O}, defined by
\begin{equation} \label{def:P}
[P(z)]^i = \begin{cases}
\max \{z^i,0 \} & \;\; i \in \sI, \\
z^i & \;\; \mbox{\rm otherwise.}
\end{cases}
\end{equation}
The matrix $D_k$ scales the free and active parts of the gradient differently, in a way that guarantees decrease in the objective function for sufficiently small positive steplengths $\alpha_k$.
Denoting a set of ``apparently-active'' components of $x_k$ by
\begin{equation}
    \label{def:Ik+}
    I_k^+ (\epsilon_k) \triangleq \{ i \in \sI \mid 0 \le x_k^i \le \epsilon_k, \; \nabla_i f(x_k) > 0 \},
\end{equation}
for small positive $\epsilon_k$, $D_k$ is assumed to be positive diagonal in the $I_k^+(\epsilon_k)$ components, that is, $D_k[i,j]=0$ \rev{if either $i$ or $j$ is  in $I_k^+(\epsilon_k)$ with $j \neq i$}, and $D_k[i,i]>0$ for all $i \in I_k^+(\epsilon_k)$.
The two-metric projection method can have rapid convergence when $f(x)$ is convex and the square submatrix of $D_k$ for the ``apparently-free" indices $i \notin I_k^+(\epsilon_k)$ is derived from the corresponding submatrix of the Hessian $\nabla^2 f(x_k)$. 
The complexity properties of this method in the setting of nonconvex $f$ are \rev{the subject of ongoing work}.

 Inspired by both two-metric gradient projection approach  and the  Newton-CG algorithm for unconstrained optimization described in \cite{Royer2019}, we propose a projected Newton-CG algorithm.
We show that the algorithm terminates within $\sO(\epsilon^{-3/2})$ iterations and outputs  an $(\epsilon,\tfrac12)$-2o point with high probability. 
In each iteration of the projected Newton-CG, we either (1) take a gradient projection step;  (2)  take a projected Newton-CG step, obtained via a capped CG procedure applied to the apparently-free components, or (3) take a projected step along a negative curvature direction of a diagonally scaled Hessian.
The operations required to calculate each type of step are well defined, and are similar to those used in \cite{Royer2019,doi:10.1137/17M1134329}.
These ``fundamental operations" are of two types: (1) a gradient calculation, and (2) computation of the product of the Hessian with an arbitrary vector --- an operation that does not require explicit computation or knowledge of the Hessian and that can be performed at roughly equivalent cost to a gradient evaluation; see \cite{griewank2008evaluating}.  
The other potentially significant computations are (1) function evaluations performed during the backtracking line searches, the number of which is bounded by an $O(\log \epsilon)$ multiple of the number of gradient evaluations, and which are usually significantly cheaper than gradient evaluations; and (2) vector operations involving vectors of length $n$ (inner products and saxpys), whose $\sO(n)$ cost is dominated by the cost of the fundamental operations for all functions of interest.
By contrast, other methods require solution of potentially expensive constrained nonconvex subproblems in each iteration \cite{birgin2018regularization,cartis2012adaptive,cartis2015evaluation,cartis2018second} and possibly explicit evaluation of Hessians and higher derivatives. \revn{These requirements have the  potential to make the computational complexity less competitive.}

\rev{Table~\ref{tab: complex.} shows iteration/evaluation complexity and operation complexity results for our algorithm (last row) and existing algorithms, based on their respective definitions of $(\epsilon,p)$-2o. 
The ``operation complexity'' results are upper bounds on the number of fundamental operations required to find an approximate solution. 
}

\begin{table}
\rev{
\footnotesize
  \caption{Complexity estimates for nonconvex optimization procedures involving bounds.}
  \label{tab: complex.}
  \begin{tabular}{clll}
  \toprule
Definition of $(\epsilon,p)$-2o$^*$ 
  & $ \begin{array}{l}
\mbox{Iteration/evaluation} \\
\mbox{Complexity }
\end{array} $  & $ \begin{array}{l}
\mbox{Operation} \\
\mbox{Complexity } (p = \tfrac12)
\end{array} $  & Ref. \\
  \hline
\eqref{eq:em2o-cartis}
   & $\sO(\epsilon^{-3})^{\star\star}$ $(p = 1)$
  & $-$ & \cite{cartis2018second}  \\
  \hline
 \eqref{eq:em2o-Haeser} & $\sO(\epsilon^{-3/2})$ $(p = 1/2)$ & $-$ & \cite{Haeser2018} \\
  \hline
 \eqref{eq:em2o-mike} & $\begin{array}{l}
 \tilde{\sO}(n\epsilon^{-1/2}+\epsilon^{-3/2})^\dagger \\
 (p = 1/2)
 \end{array}
 $
 & $ \begin{array}{l}
\tilde{\sO}(n\epsilon^{-3/4}+\epsilon^{-7/4}),  \mbox{$n$ large} \\
\tilde{\sO}(n\epsilon^{-3/2}),  \mbox{$n$ small}
\end{array} $  & \cite{10.1093/imanum/drz074} \\
 \hline
 \eqref{eq:em2o-Bian} & $\sO(\epsilon^{-3/2})$ $(p = 1/2)$ & $-$ & \cite{Bian2015} \\
 \hline
\eqref{eq:em2o} w. $\sI = \{1,\hdots,n\}$ &  $\sO(\epsilon^{-3/2})$ $(p = 1/2)$ & $\sO(\epsilon^{-3/2}\min\{n , \epsilon^{-1/4}\log(\frac{n}{\epsilon \delta}) \})$ & (here) \\
\bottomrule
\end{tabular}}

\rev{$*$: Definition of $(\epsilon,p)$-2o is based on the paper in ``Ref." but tailored to problem \eqref{opt: bc0},\eqref{def:O} with  $\sI = \{1,\hdots,n\}$.\\
$\star\star$: \rev{When $p = 1$, accuracy on the optimality condition involving Hessian is higher, leading to a higher complexity bound.} 
\\
$\dagger$: $\tilde{\sO}$ represents $\sO$ with logarithmic factors omitted.\\
}
\end{table}

Illustrative numerical experiments on nonnegative matrix factorization problems show that the projected Newton-CG algorithm has good practical performance: It contends well with gradient projection method and the log-barrier Newton-CG algorithm proposed in \cite{10.1093/imanum/drz074}, and is comparable to approaches that are specialized to this problem in relatively low dimensions.

With minor modifications (c.f. Appendix~\ref{app: 2sidebds}), the projected Newton-CG can be applied to problems with two-sided bounds, where $\Omega$ is redefined as $\{ x \in \bR^n \mid 0 \le x^i \le u^i, i \in \sI \}$, $\sI \subseteq \{ 1,2,\hdots,n \}$, with the same complexity guarantees.

\paragraph{Organization.} In Section~\ref{sec: Prelim}, we introduce some basic assumptions and definitions to be used throughout the article. 
Definitions of the approximate second-order optimal point in our work and others are discussed in Section~\ref{sec: approx-2o}. The projected Newton-CG is presented and analyzed in Section~\ref{sec: PNCG}. Section~\ref{sec: num} describes numerical experiments. Section~\ref{sec: con} contains some concluding remarks.

We include in the Appendix details of the relationship between different definitions of approximate second-order optimality, the oracles utilized in the projected Newton-CG algorithm, and extension to two-sided bounds.

 \section{Preliminaries} \label{sec: Prelim}
 
 We summarize here some notations, two assumptions used throughout the paper, and (exact) optimality conditions for \eqref{opt: bc0}, \eqref{def:O}.
 
 \paragraph{Notation.} We use subscripts for iteration numbers (usually $k$) throughout,
  and denote components of vectors by superscripts and components of matrices using square-bracket notation, with $[i,j]$ denotes the $i,j$ element. 
  We use the following notation for gradient and Hessian of $f$ at $x_k$:  
  \[
  g_k \triangleq \nabla f(x_k), \quad H_k \triangleq \nabla^2 f(x_k).
  \]
  We use $\nabla_i f(x)$ to denote the $i$th component of $\nabla f(x)$. $\diag(v)$ is a diagonal matrix with $v^i$ being its $[i,i]$ element. $\sgn(z) = 1$ if $z \ge 0$ and $\sgn(z) = -1$ otherwise. $\| \cdot \|$ denotes the 2-norm of a vector or a matrix. $c_+ \triangleq \max\{ c, 0 \}$ for a scalar $c \in \bR$. \rev{$\sI^c \triangleq \{ 1,\hdots,n \} \setminus \sI$. $P(\cdot)$ denotes the projection onto the feasible region $\Omega$.} 

\paragraph{Assumptions.}  
The following assumptions are used throughout the paper, though they are not mentioned explicitly in the statements of some lemmas.
\begin{assumption} \label{Ass: comp.lev.set}
The level set $\sL_f(x_0) \triangleq \{ x \in \bR^n  \mid x \in \Omega, \; f(x) \le f(x_0) \}$ is compact.
\end{assumption}
\begin{assumption} \label{Ass: 2Lipstz}
$f$ is twice Lipschitz continuously differentiable on an open convex set containing $\sL_f(x_0)$ and all the trial points generated by Algorithm~\ref{Alg: PNCG}.
\end{assumption}


Lipschitz constants for $f$, $\nabla f(x)$ and $\nabla^2 f(x)$ on the set described in Assumption~\ref{Ass: 2Lipstz} are denoted by $L_f$, $L_g$ and $L_H$, respectively. 
Thus, for any $x,v \in \bR^{n}$ such that $x$ and $x+v$ are in this set, we have
\begin{subequations}
\begin{align}
\label{ineq: T1}
& f(x + v) \le f(x) + L_f \| v \|, \\
\label{ineq: T2}
& f(x + v) \le f(x) + \nabla f(x)^Tv + \frac{L_g}{2} \| v \|^2, \\
\label{ineq: T3}
& f(x + v) \le f(x) + \nabla f(x)^T v + \frac{1}{2} v^T \nabla^2 f(x) v + \frac{L_H}{6} \| v \|^3.
\end{align}
\end{subequations}

\rev{Therefore, $\| \nabla f(x) \| \le L_f$ and $\| \nabla^2 f(x) \| \le L_g$ over $\sL_f(x_0)$.}

\paragraph{Optimality Conditions.}
We can write first-order optimality conditions for \eqref{opt: bc0}, \eqref{def:O} (also known as stationarity conditions) at a point $\bar x$ as follows:
\begin{align} \label{1o-stat}
\begin{aligned}
& \bar x^i \ge 0, \quad \nabla_i f(\bar x) \ge 0, \; \quad \forall i \in \sI; \\ 
& \nabla_i f(\bar x) = 0, \quad \forall i \in \sI^c \cup \{ i \in \sI \mid \bar x^i > 0 \}. 
\end{aligned}
\end{align}
A weak second-order condition for \eqref{opt: bc0}, \eqref{def:O}  is that the two-sided projection of $\nabla^2 f(\bar{x})$ onto the variables $i$ such that $\bar{x}^i>0$ or $i \in \sI^c$ is positive semidefinite, which is equivalent to
\begin{equation}
   \label{2o-stat}
z^T \nabla^2 f(\bar x) z \ge  0, \quad \forall z \in \{ z \in \bR^n \mid z^i = 0, i \in \{ i \in \sI \mid \bar x^i = 0 \} \}.
\end{equation}
This condition coincides with the usual second-order necessary condition where there are no ``degenerate'' indices, that is, indices $i \in \sI$ for which both $\bar{x}^i=0$ and $\nabla_i f(\bar{x}) = 0$. 
\rev{When such indices exist, a standard second-order necessary condition is:
\begin{align*}
z^T \nabla^2 f(\bar x) z \ge  0, \quad \forall z \in \left\{ z \in \bR^n \Big| 
\begin{array}{cc}
    z^i = 0, & \mbox{if} \; i \in \sI, \bar x^i = 0, \nabla_i f(\bar x) > 0,\\
    z^i \ge 0 & \mbox{if} \; i \in \sI, \bar x^i = 0, \nabla_i f(\bar x) = 0.
\end{array}
\right\}.
\end{align*}
However, checking this condition can be as hard as checking copositivity of a matrix, which is NP-hard.} Thus, as in previous works (such as \cite{10.1093/imanum/drz074}), we base our analysis on the less stringent condition \eqref{2o-stat}.

\section{\rev{Approximate second-order optimal points}}\label{sec: approx-2o}

In this section we give our definition of $(\epsilon,p)$-approximate second-order optimal points and compare it with similar definitions in the literature. For simplicity of notation, we use $(\epsilon,p)$-2o points to denote $(\epsilon,p)$-approximate second-order optimal points. We assume $\epsilon,p > 0$ throughout.


Our definition of an $(\epsilon,p)$-2o point is as follows.
\begin{definition}[($\epsilon,p$)-2o, \soda] \label{def:em2o}
 $x$ is an ($\epsilon,p$)-2o point of \eqref{opt: bc0},\eqref{def:O} according to $\soda$ if $x \in \Omega$ and for sets $J^+$ and $J^-$ defined by
\begin{align*}
    J^+ & \triangleq \{ i \in \sI \mid 0 \le x^i \le \sqrt{\epsilon} \}, \\
J^- & \triangleq \{1,\hdots,n\} \setminus J^+ =  \sI^c \cup \{ i \in \sI \mid x^i > \sqrt{\epsilon} \},
\end{align*}
and for  diagonal matrix $S=\diag(s)$ with $s^i= 1$ when $i \in J^-$ and $s^i = x^i$ when $i \in J^+$, we have
\begin{subequations} \label{eq:em2o}
\begin{align}\label{eq:em2o.g}
 & \| S \nabla f(x) \| \le 2\epsilon,\quad \nabla_i f(x) \ge - \epsilon^{3/4}, \, \mbox{\rm for all $i \in J^+$,} \\
 \label{eq:em2o.hess}
 & S \nabla^2 f(x) S \succeq - \epsilon^p I.
\end{align}
\end{subequations}
\end{definition}

Definition~\ref{def:em2o} is motivated by the (weak) second-order optimal conditions \eqref{1o-stat} and \eqref{2o-stat}. In fact, if we let $\epsilon = 0$, then the $(0,p)$-2o point satisfies \eqref{1o-stat} and \eqref{2o-stat} exactly. The following lemma further justifies Definition~\ref{def:em2o} and our purpose to find an $(\epsilon,p)$-2o point given small $\epsilon$. 
\begin{lemma}
Consider problem \eqref{opt: bc0},\eqref{def:O}. Suppose we have a positive scalar sequence $\{ \epsilon_k \}$ with $\epsilon_k \downarrow 0$ and vector sequence $\{ x_k \} \subseteq \Omega$ with  $x_k \rightarrow x^*$ such that $x_k$ is an $(\epsilon_k,p)$-2o point according to Definition~\ref{def:em2o}. Then $x^*$ satisfies second-order optimal conditions \eqref{1o-stat}, \eqref{2o-stat}. 
That is, for sets $\sJ_*^-$ and $\sJ_*^-$ defined by
\[
\sJ_*^- \triangleq \sI_c \cup \{ i \in \sI \mid (x^*)^i > 0 \}, \quad 
\sJ_*^+ \triangleq \{1,2,\hdots,n\} \setminus \sJ_*^-,
\]
we have 
\begin{subequations} \label{eq:xstar-stat}
\begin{align}
\label{xstar-stat1}
    (x^*)^i \ge 0, \quad  \nabla_i f(x^*) \ge 0, \quad & \forall i \in \sI; \\
    \label{xstar-stat2}
     \nabla_i f(x^*) = 0, \quad & \forall i \in \sJ_*^-; \\
    \label{xstar-stat3}
     z^T \nabla^2 f(x^*) z \ge 0, \quad & \forall z \in \{ z \in \bR^n \mid z^i = 0, \, i \in \sJ_*^+  \}.
\end{align}
\end{subequations}
\end{lemma}
\begin{proof}
Denote sets $\sJ_k^+$, $\sJ_k^-$ and diagonal matrix $\sS_k = \diag(s_k)$ which correspond to $J^+$, $J^-$, and $S$ in Definition~\ref{def:em2o} with  $x=x_k$, $\epsilon=\epsilon_k$ and $s = s_k$. Note that since $x_k \to x^*$ and $\epsilon_k \downarrow 0$, there exists $\bar k$ such that for any $k > \bar k$, we have
$\sJ_k^+ \subseteq \sJ_*^+$, $\sJ_*^- \subseteq \sJ_k^-$. 
Our claim that $x^*$ satisfies \eqref{eq:xstar-stat} is a consequence of the following four observations.
\begin{itemize}
    \item[(i)] Feasibility of $x^*$ follows from closedness of $\Omega$. 
    \item[(ii)] For any $i \in \sI$ and any $k$, either $i \in \sJ_k^+$ so $\nabla_i f(x_k) \ge -\epsilon_k^{3/4}$, or $i \in \sJ_k^-$ so $| \nabla_i f(x_k) | \le 2 \epsilon_k \implies \nabla_i f(x_k) \ge - 2 \epsilon_k$. 
By taking limits, we have $\nabla_i f(x^*) \ge 0$. 
    \item[(iii)] Fix any $i \in \sJ_*^-$. For all $k > \bar k$, we have $i \in \sJ_k^-$. Therefore, $s_k^i = 1$ and $| \nabla_i f(x_k)| \le 2 \epsilon_k$. 
By taking limits, we have $\nabla_i f(x^*) = 0$. 
    \item[(iv)] Fix any $z \in \{ z \in \bR^n \mid z^i = 0, i \in \sJ_*^+ \}$. For all $k > \bar k$, we have $i \in \sJ_k^+ \implies i \in \sJ_*^+ \implies z^i=0$, so that  $\sS_k z = z$. 
Since $z^T \sS_k \nabla^2 f(x_k) \sS_k z \ge - \epsilon_k^p \| z \|^2$ for any $k$, we have by taking limits that $z^T \nabla^2 f(x^*) z \ge 0$. 
\end{itemize}\qed
\end{proof}


We now identify  several definitions of approximate second-order optimal conditions proposed in literature and discuss their relationship. For simplicity, we assume in the rest of this section that 
\begin{align}\label{def:sI}
    \sI \triangleq \{1,2,\hdots,n \},
\end{align}
(so that $\Omega= \bR^n_+$, the nonnegative orthant).
When we refer to Definition~\ref{def:em2o} or $\soda$ in the rest of this section, we implicitly assume that \eqref{def:sI} holds.

We start from a definition in \cite{cartis2018second}, which is defined for optimization with general convex constraints and high-order optimal points. Here we tailor it to fit the scope of this paper: second-order optimal points and bound-constrained optimization: \eqref{opt: bc0}, \eqref{def:O}, \eqref{def:sI}.

\begin{definition}[\cite{cartis2018second}, \sodb]\label{def:em2o-cartis}
$x$ is an $(\epsilon,p)$-2o point of \eqref{opt: bc0}, \eqref{def:O}, \eqref{def:sI} according to $\sodb$ if $x \ge 0$ and, for some user-defined constant $\Delta_{\max}$ that is independent of $x$ and $\epsilon$, there exists $\Delta \in (0,\Delta_{\max}]$ such that
\begin{align}\label{eq:em2o-cartis}
\begin{aligned}
        \left| \mbox{globalmin}_{x + d \in \Omega, \| d \| \le \Delta} \quad  \nabla f(x)^T d  \right| &\le \Delta \epsilon, \\
    \left| \mbox{globalmin}_{x + d \in \Omega, \| d \| \le \Delta} \quad  \nabla f(x)^T d + \frac{1}{2}d^T \nabla^2 f(x) d \right| &\le \Delta^2 \epsilon^p.
\end{aligned}
\end{align}
\end{definition}

$\Delta_{\max}$ is often chosen to reduce the effort in global minimization.

The following three definitions are from \cite{Haeser2018,10.1093/imanum/drz074,Bian2015} tailored to our problem of interest. Here we let $X = \diag(x)$, $\bar X = \diag(\min\{x,{\bf 1}\})$ and ${\bf 1}$ denotes the vectors with all elements being $1$.

\begin{definition}[\cite{Haeser2018}, \sodc]\label{def:em2o-Haeser} $x$ is an $(\epsilon,p)$-2o point of \eqref{opt: bc0}, \eqref{def:O}, \eqref{def:sI} according to $\sodc$ if
\begin{align}\label{eq:em2o-Haeser}
    \begin{aligned}
    x \ge 0, \; \nabla f(x) \ge -\epsilon {\bf 1}, \; \| X \nabla f(x) \|_\infty &\le \epsilon, \\
  X \nabla^2f(x) X & \succeq - \epsilon^p I_n.
\end{aligned}
\end{align}
\end{definition}

\begin{definition}[\cite{10.1093/imanum/drz074}, \sodd]\label{def:em2o-mike} $x$ is an $(\epsilon,p)$-2o point of \eqref{opt: bc0}, \eqref{def:O}, \eqref{def:sI} according to $\sodd$ if
\begin{align}\label{eq:em2o-mike}
    \begin{aligned}
    x \ge 0, \; \nabla f(x) \ge - \epsilon {\bf 1}, \; \| \bar{X} \nabla f(x) \|_\infty & \le \epsilon,  \\
    \bar{X} \nabla^2 f(x) \bar{X} & \succeq -\epsilon^p I_n.
    \end{aligned}
\end{align}
\end{definition}

\begin{definition}[\cite{Bian2015}, \sode]\label{def:em2o-Bian} $x$ is an $(\epsilon,p)$-2o point of \eqref{opt: bc0}, \eqref{def:O}, \eqref{def:sI} according to $\sode$ if
\begin{align}\label{eq:em2o-Bian}
\begin{aligned}
      x \ge 0, \;
  \| X \nabla f(x) \|_\infty & \le \epsilon, \\ 
  X \nabla^2 f(x) X & \succeq - \epsilon^p I_n.
\end{aligned}
\end{align}
\end{definition}

The relationship between each of these definitions and second-order criticality has been discussed in the respective work. 
In order to discuss the relation between any two of these definitions including ours, we propose the following concept, which relates pairs of definitions of $(\epsilon,p)$-2o under the assumption that $x$ is confined to a compact set $\sX$.

\begin{definition}\label{def:essstr}
We say that $\sodAA$ is {\bf essentially stronger} than $\sodBB$ on $\sX$ if given any sufficiently small $\epsilon \in (0,1]$, any $(\epsilon,p)$-2o point $x \in \sX$ by $\sodAA$ is also a $(c\epsilon,p)$-2o point by $\sodBB$, where $c > 0$ is a constant independent of $\epsilon$ or $x$. We denote this relation as $\sodAA \gtrsim_{f,\sX,p} \sodBB$, simplified as $\sodAA \gtrsim \sodBB$. 
We say that $\sodAA$ and $\sodBB$ are {\bf essentially equivalent} (denoted $\sodAA \thickapprox \sodBB$) if $\sodAA \gtrsim \sodBB$ and $\sodBB \gtrsim \sodAA$.
\end{definition}

Transitivity of the relation $\gtrsim$ is shown in Lemma~\ref{def:trans}. 

Comparison and evaluation of complexity of different algorithms makes more sense if we are able to relate the guarantees on the points they produce according to the relations in Definition~\ref{def:essstr}. In fact, if we care most about the complexity as a function of the accuracy parameter $\epsilon$, Definition~\ref{def:essstr} is natural and intuitive due to the following theorem.
\begin{theorem}
Given any $\epsilon > 0$ sufficiently small, suppose that an algorithm can find an $(\epsilon,p)$-2o point $x \in \sX$ by $\sodAA$ in $\sO(\epsilon^{-q})$ iterations ($q > 0$) and $\sodAA \gtrsim \sodBB$. Then the algorithm can also locate an $(\epsilon,p)$-2o point by $\sodBB$ in $\sO(\epsilon^{-q})$ iterations.
\end{theorem}
\begin{proof}
Since $\sodAA \gtrsim \sodBB$, there is a constant $c>0$ such that for all $\epsilon > 0$ sufficiently small, an $(\epsilon/c,p)$-2o point by $\sodAA$ is an $(\epsilon,p)$-2o point by $\sodBB$. 
By assumption, the algorithm can an locate $(\epsilon/c,p)$-2o point by $\sodAA$ in $\sO((\epsilon/c)^{-q}) = \sO(\epsilon^{-q})$ number of iterations. The result follows. \qed
\end{proof}



We can now clarify several pairwise relations between the Definitions~\ref{def:em2o}-\ref{def:em2o-Bian}. The proof of the following result appears in Appendix~\ref{app:2o}.
\begin{theorem} \label{thm: defcomp}
Suppose that $\sX$ is a compact set. Then we have the following.
\be
\item[(1)] $\sodb\gtrsim\sodc$.
\item[(2)] $\sodc\thickapprox \sodd$.
\item[(3)] $\sodd \gtrsim \sode$.
\item[(4)] $\soda \gtrsim \sode$.
\ee
\end{theorem}


The assumption in Theorem~\ref{thm: defcomp} on compactness of $\sX$ is mild. In fact, many works in literature assume that the iterates generated by their algorithms lie in a compact region, for example, the sublevel set of the objective function.
By Theorem~\ref{thm: defcomp}, we have the following relation chart of Definition~\ref{def:em2o}-Definition~\ref{def:em2o-Bian}:
\begin{align*}
        \sodb \gtrsim \sodc  \thickapprox \sodd \gtrsim \sode, \quad 
    \soda  \gtrsim \sode.
\end{align*}
Note that each $\gtrsim$ relation above is probably strict. 
For example, $\sodb$ considers the global minimum of the first-order and second-order Taylor expansions of $f$ over a small trust region, 
while $\sodc$ (in fact all other definitions) is only closely related to the {\it weak} second-order necessary conditions \eqref{1o-stat},\eqref{2o-stat} for $x$ being a local minimal point. 
$\sode$ is weaker than others since it does not offer an appropriate lower bound on $\nabla_i f(x)$ when $x^i = 0$. In fact, the relation between $\sode$ and second-order criticality is also weaker than others. Unfortunately, we cannot describe by $\gtrsim$ the relation between our definition ($\soda$) with definitions other than $\sode$. On one hand, the condition $\nabla_i f(x) \ge -\epsilon^{-3/4}, i \in J^+$ in $\soda$  is weaker; on the other hand, the condition $\| S \nabla f(x) \| \le 2 c \epsilon$ is strong and cannot be implied by other $(\epsilon,p)$-2o definitions for any constant $c$ independent of $\epsilon$. An illustrative example is given in Appendix~\ref{app:2o}, Example~\ref{eg1}. 

\revn{During the review process,  we found that the definition used in \cite{nouiehed2020trust} is also relevant. When tailored to the scope in this paper (see Definition~\ref{def: nouiehed2020trust} in Appendix~\ref{app:2o}), it can be placed between $\sodb$ and $\sodc$ (see Theorem~\ref{thm: nouiehed2020trust} Appendix~\ref{app:2o}).}

\section{Projected Newton-CG method and its complexity} \label{sec: PNCG}

We now describe a projected Newton-CG algorithm to find an $(\epsilon,\tfrac12)$-2o point according to Definition~\ref{def:em2o} for problem \eqref{opt: bc0}, \eqref{def:O}, and analyze its complexity properties.


\subsection{Description of the Algorithm}

Given the sequence of iterates $\{ x_k \}$ and a positive scalar
sequence $\{ \epsilon_k \}$ we define the following index sets inspired by the two-metric projection method \eqref{eq:2mgp}, \eqref{def:Ik+}:
\begin{equation} \label{eqLdefJk}
\begin{aligned}
  J_k^+  & \triangleq \{ i \in \sI \mid 0 \le x_k^i \le \epsilon_k \}, \\
  J_k^- & \triangleq \{1,\hdots,n \} \setminus J_k^+ = \sI^c \cup \{ i \in \sI \mid
x_k^i > \epsilon_k \}.
\end{aligned}
\end{equation}
Let $g^-_k$, $H^-_k$ be the subvector and square submatrix of $g_k$ and $H_k$, resp., corresponding to index set $J_k^-$. 
Similarly, we use $g^+_k$ and $H^+_k$ for the subvector and square submatrix of $g_k$ and $H_k$, resp., corresponding to index set $J_k^+$.
For search direction $d_k$, denote $d_k^-$ and $d_k^+$ in the same fashion.
Define the scaling vector  $s_k$ and diagonal scaling matrix $S_k$ as follows:
\begin{equation} \label{eq:Sk.def}
 s^i_k \triangleq \begin{cases}  x_k^i, \; & i \in J_k^+ \\
     1, \; & i \in J_k^-
     \end{cases}, \quad
     S_k \triangleq \diag(s_k).
 \end{equation}
We can then define the projected Newton-CG algorithm as Algorithm~\ref{Alg: PNCG}.
 
\begin{algorithm} 
\caption{Projected Newton-CG (PNCG)}\label{Alg: PNCG}
\begin{algorithmic}
\STATE {\bf (Initialization)} Choose an initial point $x_0 \ge 0$, tolerance $\epsilon_g > 0$, scalar sequence $\{ \epsilon_k \}$ with $\epsilon_k  \in (0,1)$ for all $k$, backtracking parameters $\theta \in (0,1)$, accuracy parameter $\zeta \in (0,1)$, step acceptance parameter $\eta \in \left(0, \frac{1- \zeta}{2} \right)$.  
\FOR{$k=0,1,2,\dotsc$}
\IF{$J_k^+ \neq \emptyset$ \AND ( $g_k^i < -\epsilon_k^{3/2}$ for some $ i \in J_k^+$ \OR $\| S^+_k g^+_k \| > \epsilon_k^2$ )}
  \STATE {\bf (Gradient Projection step)} Let  $ d_k := - g_k$;
   \STATE Let $\tilde m_k$ be the smallest nonnegative integer $m$ such that 
  \[
  f( P(x_k+ \theta^m d_k) ) < f( x_k ) - \frac{1}{2} (x_k - P(x_k+\theta^m d_k))^T g_k;
  \]
\STATE Let $ x_{k+1} := P(x_k+ \theta^{\tilde m_k} d_k) $; 
\ELSIF{ $J_k^- \neq \emptyset$ \AND $\| g^-_k \| > \epsilon_g$}
\STATE{{\bf (Newton-CG step)} Call Algorithm~\ref{alg:ccg} (Capped CG, \rev{Appendix~\ref{app: capped-CG}}) with $H := H^-_k$, $\epsilon := \epsilon_k$, $g := g^-_k$, accuracy parameter $\zeta$ and upper bound $M$ on Hessian norm (if provided). Obtain outputs $t \in \bR^{|J_k^-|}$ and $d{\_\rm type}$;}
\IF{$d\_{\rm type} = {\rm NC}$}
\STATE Let $d^-_k := -{\rm sgn}(t^T g^-_k)
  \frac{ | t^T H^-_k t | }{ \| t \|^2 } \frac{t}{\| t \|} $; (Negative curvature direction) 
\ELSE
\STATE Let $d^-_k := t$; (Approx. solution to reduced Newton equations) 
\ENDIF 
\STATE Let $d^+_k=0$ (Complete $d_k$ with zeros in the active components)
\STATE Let $m_k$ be smallest nonnegative integer $m$ such that 
\[ 
f( P(x_k+\theta^m d_k) ) < f(x_k) - \eta \theta^{2m} \epsilon_k \| d_k \|^2;
\]
\STATE Let $\alpha_k := \theta^{m_k}$, $x_{k+1} := P(x_k+\theta^{ m_k}d_k)$;
\ELSE 
\STATE Call Procedure~\ref{alg:meo} (Minimum Eigenvalue Oracle (MEO), \rev{Appendix~\ref{app: MEO}}) with $H := S_k H_k S_k$, $\epsilon := \epsilon_k$ and the upper bound of norm of $H$ if known. 
  \IF{Procedure~\ref{alg:meo} certifies that $S_k H_k S_k \succeq - \epsilon_k I$}
  \STATE STOP and output $x_k$;
  \ENDIF
  \STATE {\bf (Negative curvature step)} Let $d_k := - \sgn( g_k^T S_k d ) \cdot | d^T S_k H_k
  S_k d | \cdot d $, where $d$ is the output of Procedure~\ref{alg:meo};
  \STATE Let $\bar m_k$ be the smallest nonnegative integer $m$ such that
  \[
  f( P(x_k + \theta^m S_k d_k ) ) < f(x_k) - \eta \theta^{2m} \| d_k \|^3.
  \]
  \STATE Let $x_{k+1} := P(x_k + \theta^{\bar m_k} S_k d_k ) $; 
\ENDIF
\ENDFOR
\end{algorithmic}
\end{algorithm}


\paragraph{Elements of Algorithm~\ref{Alg: PNCG}.} 
\rev{As in the two-metric projection method \eqref{eq:2mgp}, our method starts each iteration by partitioning the components of $x$ into the ``apparently-free" and ``apparently-active" indices based on their proximity to the boundary and a threshold parameter $\epsilon_k$. Then one of three types of steps is taken. For all such steps, backtracking in combination with projection onto the feasible set is used to determine an appropriate steplength.
\begin{itemize}
    \item {\bf Gradient projection step:} If examination of the gradient components corresponding to the apparently-active components indicate that a significant improvement in $f$ can be obtained by taking a standard gradient projection step, such a step is taken.
    \item {\bf Newton-CG step on apparently-free components:} When the gradient corresponding to the apparently-free components is above the threshold $\epsilon_g$, the Capped CG procedure (c.f. Appendix~\ref{app: capped-CG}) is called to either find an approximate Newton step in these components, or else return a direction of negative curvature. Only the apparently-free components are modified in a step of this type.
    \item {\bf Scaled negative curvature step (full-dimensional):} When neither of the two types of steps defined above is deemed appropriate, the current iterate $x_k$ satisfies the approximate optimality conditions of Definition~\ref{def:em2o}, except for the condition \eqref{eq:em2o.hess} on the scaled Hessian. We therefore check this condition and, if it is not satisfied, find a scaled negative curvature step that will lead to a significant decrease in $f$. While the other type of negative curvature step (obtained from Capped CG) changes only the apparently-free components, this scaled negative curvature step changes {\em all} components, in general. We believe that this type of step will rarely be taken; most instances of negative curvature will be detected during computation of the Newton-CG step.
\end{itemize}
}

\paragraph{Connections to known methods for bound-constrained and unconstrained optimization.}
\rev{The way in which Algorithm~\ref{Alg: PNCG} combines Newton-CG steps with gradient projection steps is inspired in part by
Mor{\'e} and Toraldo \cite{More91GPCG}, who use CG iterations applied to the Newton system to ``explore" a face of the feasible orthant and gradient projection to move to a new face. However, \cite{More91GPCG} addresses only convex quadratic problems and has no complexity analysis.}



\rev{There are obvious connections between Algorithm~\ref{Alg: PNCG} and the Newton-CG methods for unconstrained nonconvex optimization described in \cite{Royer2019} and \cite{doi:10.1137/17M1134329}. 
The latter methods make use of Capped CG procedures (where the "cap" refers to an implicit bound on the number of CG iterations allowed at each invocation), as well as negative curvature directions and backtracking line searches.
We leverage the similarities by using the same ``subroutines" for Capped CG and negative curvature detection as in \cite{Royer2019}; these methods are stated for completeness in Appendices~\ref{app: capped-CG} and \ref{app: MEO}, along with their key properties.
However, the modifications required to adapt the approach of \cite{Royer2019} to handle bound constraints, in a way that allows complexity results to be proved, are significant and non-obvious.
For one thing, we cannot simply project the approximate Newton step onto the feasible region, as this  may not yield descent even for convex $f$; see \cite[Section 1.5]{bertsekas2014constrained}. Indeed, Bertsekas proposed the two-metric gradient projection approach precisely to deal with this issue.
Essentially, the proximity of iterates $x_k$ to the boundary of the feasible set $\Omega$ and the use of projection inhibit steps in ways that may prevent the ``significant decrease" in objective $f$ required at each iteration to prove complexity. 
We need to use scaling of steps and Hessians, modified steplength acceptance criteria, and novel partitions of the set of components to overcome this potential hazard.
Differences with prior work, particularly the unconstrained Newton-CG approach of \cite{Royer2019}, can be summarized as follows.
\begin{enumerate}
    \item Our partition of $\{1,2,\dotsc,n\}$ into apparently-active and apparently-free parts \eqref{eqLdefJk} differs from standard two-metric gradient projection in not considering the sign of the gradient.
    \item We use a gradient projection step in certain conditions; devising these conditions in such a way that the step yields the significant improvement in $f$ required by our complexity analysis (see Lemma~\ref{lm: gradproj}) is somewhat intricate.
    \item We utilize a different sufficient decrease criterion for the Newton-CG step from the one in \cite{Royer2019}, and this step takes place only in the subspace of apparently-free  variables. The analysis in proofs of Lemmas~\ref{lm: ncdec} and \ref{lm: soldec} is similar to that of corresponding results in \cite{Royer2019}, but takes the presence of bound constraints in the apparently-free variables into account.
    \item We compute the full-dimensional negative curvature direction on a {\em diagonally scaled} version of the Hessian, and need a scaled direction and a different sufficient decrease condition from \cite{Royer2019}.
\end{enumerate}
}


\subsection{Complexity of Algorithm~\ref{Alg: PNCG}}

The following  four results --- Lemmas~\ref{lm: gradproj} to \ref{lm: meostep} --- prove a lower bound on the amount of decrease in $f$ at a single iteration in each of the following four cases. (We assume that Assumptions~\ref{Ass: comp.lev.set} and \ref{Ass: 2Lipstz} hold with $\Omega$ in \eqref{def:O} for all these results, although we do not mention them in the statement of each result.)
\begin{itemize}
    \item[(i)] A gradient projection step is taken (Lemma~\ref{lm: gradproj});
    \item[(ii)] The Newton-CG step is triggered and the Capped CG algorithm returns $d\_{\rm type} = {\rm NC}$, resulting in a negative curvature step involving the apparently-free components (Lemma~\ref{lm: ncdec});
    \item[(iii)] The Newton-CG step is triggered and the  Capped CG algorithm returns $d\_{\rm type} = {\rm SOL}$, resulting in a Newton-like step (Lemma~\ref{lm: soldec});
    \item[(iv)] The MEO procedure returns a negative curvature direction instead of a certificate of optimality, and a negative curvature step is taken (Lemma~\ref{lm: meostep}).
\end{itemize}
We state and prove these results without further elaboration.

\begin{lemma} \label{lm: gradproj}
Suppose that $J_k^+ \neq \emptyset$ at iteration $k$, and that  $ g_k^i < -\epsilon_k^{3/2} $ for some $i \in J_k^+$ or $ \| S^+_k g^+_k \| > \epsilon_k^2 $, so that a projected gradient step is taken. Then
\begin{align*}
f(x_k) - f(x_{k+1}) > \frac{1}{4} \min\{ \theta/L_g, 1 \} \epsilon_k^3.
\end{align*}
\end{lemma}
\begin{proof}
If $g_k^i < -\epsilon_k^{3/2}$ for some $i \in J_k^+$ or $\| S^+_k g^+_k \| > \epsilon_k^2$ at the gradient projection step, then for any steplength $\beta > 0$, at least one of two cases occurs. In the first case of $g_k^i < -\epsilon_k^{3/2}$ for some $i \in J_k^+$, we have
\begin{equation}
\label{ineq: dec.1}
g_k^i < -\epsilon_k^{3/2} \implies ( g_k^i )^2 > \epsilon_k^3 \implies ( x_k^i - ( x_k^i - \beta g_k^i )_+) g_k^i = \beta (g_k^i)^2 > \beta \epsilon_k^3.
\end{equation}
In the second case, we have
\begin{align*}
\| S^+_k g^+_k \|^2 > \epsilon_k^4 & \implies \sum_{i \in J_k^+} (x_k^i)^2 (g_k^i)^2 > \epsilon_k^4 \\
& \implies \sum_{i \in J_k^+, \beta g_k^i \le x_k^i} (x_k^i)^2 (g_k^i)^2 + \sum_{i \in J_k^+, \beta g_k^i > x_k^i} (x_k^i)^2 (g_k^i)^2 > \epsilon_k^4.
\end{align*}
Therefore, either
\begin{align*}
    \sum_{i \in J_k^+, \beta g_k^i \le x_k^i} (x_k^i)^2 (g_k^i)^2 \ge \epsilon_k^4/2 \overset{ (x_k^i \le \epsilon_k, \forall i \in J_k^+) }{\implies}  \sum_{i \in J_k^+, \beta g_k^i \le x_k^i} (g_k^i)^2 \ge \epsilon_k^2/2,
\end{align*}
or
\begin{align*}
    \sum_{i \in J_k^+, \beta g_k^i > x_k^i} (x_k^i)^2 (g_k^i)^2 \ge \epsilon_k^4/2 \implies \sum_{i \in J_k^+, \beta g_k^i > x_k^i} x_k^i g_k^i \ge \epsilon_k^2/\sqrt{2}.
\end{align*}
Thus in this case, we have
\begin{align}
\notag
\sum_{i \in J_k^+}   ( x_k^i - (x_k^i - \beta g_k^i)_+ )g_k^i  
& = \sum_{i \in J_k^+, \beta g_k^i \le x_k^i } \beta (g_k^i)^2 + \sum_{i \in J_k^+, \beta g_k^i > x_k^i }  x_k^i g_k^i \\
\notag
& \ge \min\{ \beta/2, 1/\sqrt{2} \} \epsilon_k^2 \\
\label{ineq: dec.2}
& \overset{(\epsilon_k
< 1)}{>}  \min\{ \beta/2, 1/\sqrt{2} \} \epsilon_k^3.
\end{align}
By noting $g_k^i(x_k^i - (x_k^i - \beta g_k^i)_+
) \ge 0$ for any $i \in \sI$, we have for any $\beta > 0$ that
\begin{align}
\notag
g_k^T ( x_k -  P(x_k-\beta g_k) ) & = \sum_{i \in \sI} g_k^i ( x_k^i - (x_k^i - \beta g_k^i)_+ ) + \sum_{i \in \sI^c} \beta (g_k^i)^2 \\
\notag
& \ge \sum_{i \in J_k^+} g_k^i ( x_k^i - (x_k^i - \beta g_k^i)_+ ) \\
\label{ineq: dec.3}
& \overset{ \eqref{ineq: dec.1}, \eqref{ineq: dec.2} }{>} \min\{ \beta/2,1/\sqrt{2} \} \epsilon_k^3.
\end{align}
Note for any $0 < \beta < \frac{1}{L_g}$, where $L_g$ is the Lipschitz constant of $\nabla f$, we have
\begin{align*}
f( P(x_k-\beta g_k) ) & \le 
f(x_k) - g_k^T ( x_k - P(x_k - \beta g_k) ) + \frac{L_g}{2} \| x_k - P(x_k - \beta g_k) \|^2 \\
& \le f(x_k) - g_k^T ( x_k - P(x_k - \beta g_k) ) + \frac{L_g}{2} \beta g_k^T (x_k - P(x_k - \beta g_k) ) \\
& < \rev{ f(x_k) - g_k^T ( x_k - P(x_k - \beta g_k) ) + \frac{1}{2} g_k^T (x_k - P(x_k - \beta g_k) ) } \\
& = f( x_k ) - \frac{1}{2} g_k^T ( x_k - P(x_k - \beta g_k) ) ,
\end{align*}
where the second inequality holds because $ (u - v)^T(P(u) - P(v)) \ge \| P(u) - P(v) \|^2$ for any $u, v \in \bR^{n}$, and the third inequality holds because $ \beta < 1/L_g$ \rev{and $g_k^T( x_k - P(x_k - \beta g_k) ) > 0$ by \eqref{ineq: dec.3}.} Therefore, by the line search rule, $\tilde m_k < +\infty$ and $\theta^{\tilde{m}_k} \ge \min\left\{ \frac{\theta}{L_g}, 1 \right\}$.
Thus, by the lower bound for $\theta^{\tilde m_k}$, the bound \eqref{ineq: dec.3}, and the backtracking line search mechanism, we have
\begin{align*}
f(x_k) - f(x_{k+1}) & > \frac{1}{2} g_k^T(x_k - P(x_k- \theta^{\tilde m_k} g_k)) > \frac{1}{4} \min\left\{ \theta/ L_g, 1 \right\} \epsilon_k^3.
\end{align*}\qed
\end{proof}

\begin{lemma} \label{lm: ncdec}
  Suppose that at iteration $k$, a Newton-CG step is triggered and that Algorithm~\ref{alg:ccg}  returns $d\_{\rm type} = {\rm NC}$. 
  Then we have
  $m_k < +\infty$ and
 \begin{align*}
 f(x_k) - f(P(x_k+\alpha_k d_k)) > \cnc \epsilon_k^3, 
 \end{align*}
  where $\cnc \triangleq \eta \min \left\{ \frac{ (3 - 6\eta)^2 \theta^2 }{ L_H^2 }, \theta^2 \right\}$.
\end{lemma}
\begin{proof}
For the Newton-CG step, if $ \| \alpha d_k \| \le \epsilon_k $ for some $\alpha>0$,
then $\| \alpha d^-_k \|_\infty = \| \alpha d_k \|_\infty \le
\epsilon_k $ and $P(x_k+\alpha d_k) = x_k + \alpha d_k$. From \eqref{ineq:
  T3}, we have
\begin{align} \nonumber
f( P(x_k+ \alpha d_k) ) & = f(x_k + \alpha d_k) \\
\label{ineq: pp0} 
& \le f(x_k) + \alpha g_k^T d_k + \frac{ \alpha^2 }{ 2 } d_k^T H_k d_k + \frac{L_H}{6} \alpha^3 \| d_k \|^3.
\end{align}
Since $d\_{\rm type} = {\rm NC}$, we have that $(d^-_k)^T g^-_k \le 0$, and from
Lemma~\ref{lm: capCG}(let $\bar d = d_k^-$,$\epsilon = \epsilon_k$) that $ \frac{(d^-_k)^T H^-_k d^-_k }{ \| d^-_k \|^2 } = - \| d^-_k \| \le -\epsilon_k $. Then for any $ 0 < \alpha < \frac{3 - 6\eta}{ L_H }$,
\begin{align}
\notag
& f(x_k) + \alpha g_k^T d_k + \frac{\alpha^2}{2} d_k^T H_k d_k + \frac{L_H}{6} \alpha^3 \| d_k \|^3 \\
\notag
& = f(x_k) + \alpha (g^-_k)^T d^-_k + \frac{\alpha^2}{2} (d^-_k)^T H^-_k d^-_k + \frac{L_H}{6} \alpha^3 \| d^-_k \|^3 \\
\notag
& \le f(x_k) - \frac{\alpha^2}{ 2 } \| d^-_k \|^3 + \frac{L_H}{6} \alpha^3 \| d^-_k \|^3 \\
\label{ineq: pp6}
& < f(x_k) - \eta \alpha^2 \| d^-_k \|^3 \le f(x_k) - \eta \alpha^2 \epsilon_k \| d_k \|^2. 
\end{align}
Then, by leveraging \eqref{ineq: pp0} and \eqref{ineq: pp6}, we have
that if $\alpha < \min\left\{ \frac{3 - 6\eta}{L_H},
\frac{\epsilon_k}{ \| d_k \| } \right\}$, then $f(P(x_k+\alpha d_k)) <
f(x_k) - \eta \alpha^2 \epsilon_k \| d_k \|^2$. Therefore,
backtracking will terminate when $\alpha_k$ drops below $\min\left\{
\frac{ 3 - 6\eta }{L_H}, \frac{\epsilon_k}{ \| d_k \| } \right\}$, if not earlier. 
Further, because of the backtracking mechanism, $\alpha_k$  cannot be less than $\theta$ times this value. 
As a result, we have
\begin{align*}
\alpha_k \ge & \min\left\{ \theta \min\left\{ \frac{3 - 6\eta}{L_H}, \frac{\epsilon_k}{ \| d_k \| } \right\}, 1 \right\} \\
\implies \alpha_k \| d_k \| \ge & \min \left\{ \frac{ (3-6\eta)\theta \| d_k \|}{L_H}, \theta \epsilon_k, \| d_k \| \right\} \\
\overset{ ( \| d_k \| \ge \epsilon_k  ) }{\ge} & \min \left\{ \frac{(3-6\eta)\theta}{L_H}, \theta,1 \right\} \epsilon_k \\
\implies \alpha_k^2 \epsilon_k \| d_k \|^2 \ge & \min \left\{ \frac{(3-6\eta)^2 \theta^2 }{L_H^2}, \theta^2 \right\} \epsilon_k^3.
\end{align*}
Also, $ \| d_k \| = \| d^-_k \| = \frac{ | (d^-_k)^T H^-_k d^-_k | }{ \| d^-_k \|^2 } \le \| H^-_k \|_2 \le \| H_k \|_2 \le L_g $ and
\begin{align*}
& \alpha_k \ge \min\left\{ \theta \min\left\{ \frac{3 - 6 \eta}{L_H}, \frac{\epsilon_k}{ \| d_k \| } \right\}, 1 \right\} \overset{ ( \| d_k \| \le L_g  ) }{\ge} \min\left\{ \frac{ (3-6\eta) \theta }{L_H}, \frac{\theta  \epsilon_k}{ L_g }, 1 \right\} \\
\implies & m_k = \log_\theta \alpha_k \le \max \left\{ \log_\theta \left( \frac{(3-6\eta) \theta}{ L_H } \right), \log_\theta \left( \frac{ \theta \epsilon_k }{ L_g } \right), 0 \right\},
\end{align*}
verifying that $m_k$ is finite and completing the proof.\qed
\end{proof}

\begin{lemma}  \label{lm: soldec}
 Suppose that at iteration $k$, a Newton-CG step is triggered. Moreover, Algorithm~\ref{alg:ccg} returns $d\_{\rm type} = {\rm SOL}$.
  Then $m_k < +\infty$ and
\begin{equation} \label{ineq: sd-sol}
f(x_k) - f(P(x_k+\alpha_k d_k))  > \csol \min\{ \| \nabla f( P(x_k+\alpha_k d_k) ) \mid_{J_k^-} \|^2 \epsilon_k^{-1}, \epsilon_k^3 \},
\end{equation}
where 
\rev{
\[
 \csol \triangleq 
    \eta \min\left\{ \frac{4}{25 + 8 L_H}, \theta^2, \frac{9 (1-\zeta - 2\eta)^2 \theta^2}{L_H^2} , \frac{ (1-\zeta)^2\theta^2}{ (L_H/3 + 2\eta)^2 } \right\}.
\]
}
\end{lemma}
\begin{proof}
  Define
\begin{align*}
    & l_k \triangleq \min \left\{ l \in \mathbb{N} \mid \theta^{l} \| d_k \| \le \epsilon_k \right\} \\
    & j_k \triangleq \\
    & \min \left\{ j \ge l_k, j \in \mathbb{N} \mid  \theta^j  g_k^T d_k + \frac{ \theta^{2j} }{2}  d_k^T  H_k  d_k + \frac{L_H \theta^{3j} }{6} \| d_k \|^3 < - \eta \theta^{2j} \epsilon_k \| d_k \|^2 \right\}.
\end{align*}
  Then from \eqref{ineq: pp0} and the definition of $j_k$, we have
  that 
  \[
  f(P(x_k+\theta^{j_k} d_k)) < f(x_k) - \eta \theta^{2 j_k}
  \epsilon_k \| d_k \|^2.
  \]
Therefore, by the definition of $m_k$ in Algorithm~\ref{Alg: PNCG}, it follows that $m_k \le j_k$. 
By Lemma~\ref{lm: capCG}($d = d_k^-$, $g = g_k^-$, $\epsilon=\epsilon_k$), we have
\begin{align*}
\| d_k \| = \| d^-_k \| \le 1.1\epsilon_k^{-1} \| g^-_k \| \le 1.1\epsilon_k^{-1} \| g_k \| \le 1.1 \epsilon_k^{-1} L_f,
\end{align*}
so that
\begin{equation} \label{eq:lk.bd}
l_k \le \left[ \log_\theta \left( \frac{\epsilon_k}{ \| d_k \| } \right) \right]_+ + 1 \le \left[ \log_\theta \left( \frac{\epsilon_k^2}{ 1.1 L_f } \right) \right]_+ + 1.
\end{equation}
According to Lemma~\ref{lm: capCG} (with $d = d_k^-$, $H = H_k^-$, $g = g_k^-$, $\epsilon=\epsilon_k$), we have that
\begin{subequations}
\begin{align}
\label{ineq1: royer19}
(d_k^-)^T ( H_k^- + 2 \epsilon_k I ) d_k^- & \ge \epsilon_k \| d_k^- \|^2,\\
\label{ineq2: royer19}
\| r_k^- \| & \le \frac{1}{2} \epsilon_k \zeta \| d_k^- \|,
\end{align}
\end{subequations}
where $r_k^- \triangleq (H_k^- + 2 \epsilon_k I) d_k^- + g_k^-$. Then,
\begin{align} \label{ineq: CR} 
\notag
& \theta^j (g_k^-)^T d_k^- + \frac{ \theta^{2j} }{ 2 } (d_k^-)^T H_k^- d_k^- + \frac{ L_H \theta^{3j} }{ 6 } \| d_k^- \|^3 \\
\notag
& = - \theta^j ( H_k^- d_k^- + 2 \epsilon_k d_k^- - r_k^-  )^T d_k^- + \frac{ \theta^{2j} }{ 2 } (d_k^-)^T H_k^- d_k^- + \frac{ L_H \theta^{3j} }{ 6 } \| d_k^- \|^3 \\
\notag
& = - \theta^j \left( 1 - \frac{\theta^j}{2} \right) (d_k^-)^T ( H_k^- + 2 \epsilon_k I ) d_k^- - \epsilon_k \theta^{2j} \| d_k^- \|^2 - \theta^j (r_k^-)^T d_k^- \\
\notag
& + \frac{ L_H \theta^{3j} }{ 6 } \| d_k^- \|^3 \\
\notag
&  \overset{ \eqref{ineq1: royer19}}{\le} - \theta^j \left( 1 - \frac{\theta^j}{2} \right)  \epsilon_k \| d_k^- \|^2 + \theta^j \| r_k^- \| \| d_k^- \| + \frac{ L_H \theta^{3j} }{ 6 } \| d_k^- \|^3 \\
\notag
& \overset{ \eqref{ineq2: royer19} }{\le} - \frac{\theta^j}{2}  \epsilon_k \| d_k^- \|^2  + \frac{\theta^j}{2} \epsilon_k \zeta \| d_k^- \|^2 + \frac{ L_H \theta^{3j} }{ 6 } \| d_k^- \|^3 \\
& =  - \frac{ \theta^j }{2} (1-\zeta) \epsilon_k \| d_k^- \|^2 + \frac{L_H \theta^{3j} }{6} \| d_k^- \|^3.
\end{align}
It can be verified that for any $j \ge \left[ \log_\theta \left( \frac{ (1- \zeta) \epsilon_k }{ \eta \epsilon_k + \sqrt{ \eta^2 \epsilon_k^2 + 1.1 L_H(1-\zeta) L_f/3 } } \right) \right]_+ + 1 $, we have
\begin{align*}
& \theta^j < \frac{ (1- \zeta) \epsilon_k }{ \eta \epsilon_k + \sqrt{ \eta^2 \epsilon_k^2 + 1.1 L_H(1-\zeta) L_f/3 } } \\
\overset{ ( \| d^-_k \| \le 1.1 \epsilon_k^{-1} L_f ) }{\implies} & \theta^j < \frac{ (1- \zeta) \epsilon_k }{ \eta \epsilon_k + \sqrt{ \eta^2 \epsilon_k^2 +  L_H(1-\zeta) \epsilon_k \| d^-_k \|/3 } }.
\end{align*}
It then follows from the quadratic formula applied to the  following quadratic inequality\footnote{\revn{if $a > 0$, then $z \ge 0$ and $az^2 + bz+c  < 0$ together are equivalent to $ 0 \le z <  \frac{-2c}{b + \sqrt{b^2-4ac}}$; if $a = 0$ and $b > 0$, then the equivalence still holds trivially.}} in $\theta^j$,
\begin{alignat*}{2}
& \frac{L_H \| d^-_k \| }{6} \theta^{2j} + \eta \epsilon_k \theta^j - \frac{(1-\zeta)\epsilon_k}{2} && < 0 \\
\implies & - \frac{ \theta^j }{2} (1-\zeta) \epsilon_k \| d^-_k \|^2 + \frac{L_H}{6} \theta^{3j} \| d^-_k \|^3 && < - \eta \theta^{2j} \epsilon_k \| d^-_k \|^2 \\
\overset{ \eqref{ineq: CR} }{\implies} & \theta^j (g^-_k)^T d^-_k + \frac{ \theta^{2j} }{2} (d^-_k)^T H^-_k d^-_k + \frac{L_H \theta^{3j}}{6} \| d^-_k \|^3 &&< - \eta \theta^{2j} \epsilon_k \| d^-_k \|^2 \\
\implies & \theta^j  g_k^T d_k + \frac{ \theta^{2j} }{2}  d_k^T H_k d_k + \frac{L_H \theta^{3j}}{6} \| d_k \|^3 &&< - \eta \theta^{2j} \epsilon_k \| d_k \|^2.
\end{alignat*}
\rev{Then by the definitions of $j_k$ and $l_k$ together with \eqref{eq:lk.bd}, we have}
\begin{align*}
    & j_k \le 1 + \\
    & \max\left\{ \left[ \log_\theta \left(
  \frac{\epsilon_k^2}{ 1.1 L_f } \right) \right]_+, \left[ \log_\theta
  \left( \frac{ (1- \zeta) \epsilon_k }{ \eta \epsilon_k + \sqrt{
      \eta^2 \epsilon_k^2 + 1.1 L_H(1-\zeta) L_f/3 } } \right)
  \right]_+ \right\},
\end{align*}
  which is also an upper bound for $m_k$.

Next, we derive the lower bound for $\alpha_k^2 \epsilon_k \| d_k \|^2$ which, when scaled by $\eta$,  is the required amount of decrease in $f$.  
We consider four cases.

\medskip

\noindent{\bf Case 1.} $ j_k = l_k = 0 $. 
In this case we have $m_k = 0$, $\alpha_k = 1$, and $ \| d^-_k \| = \| d_k \| \le \epsilon_k$. 
Therefore, $ x_k^i + d_k^i \ge 0, \forall i \in \sI \cap J_k^- \implies P(x_k+\alpha_k d_k) = x_k + d_k$.  
Then we have
\begin{align*}
    \| \nabla f(P(x_k+\alpha_k d_k) ) \mid_{J_k^-} \| & = \| \nabla f(x_k + d_k) \mid_{J_k^-} \| \\
    & = \| \nabla f(x_k + d_k) \mid_{J_k^-} - g_k^- + g_k^- \|\\
    & = \| \nabla f(x_k + d_k) \mid_{J_k^-} - g_k^- - H_k^- d_k^- - 2 \epsilon_k d_k^- + r_k^- \| \\
    & \le \frac{L_H}{2} \| d_k^- \|^2 + 2\epsilon_k \| d_k^- \| + \| r_k^- \| \\
    & \overset{\eqref{ineq2: royer19}}{\le} \frac{L_H}{2} \| d_k^- \|^2 + \frac{4 + \zeta}{2}\epsilon_k \| d_k^- \| \\
    & \rev{\overset{(\zeta<1)}{\le} \frac{L_H}{2} \| d_k^- \|^2 + \frac{5}{2}\epsilon_k \| d_k^- \|}.
\end{align*}
By applying the \rev{quadratic formula to the inequality above (which involves a quadratic in $\| d^-_k \|$), we obtain
\begin{align*}
\| d^-_k \| & \ge \frac{ - \frac{5}{2} + \sqrt{ \frac{25}{4} + 2 L_H \| \nabla f(P(x_k+\alpha_k d_k) ) \mid_{J_k^-} \|/\epsilon_k^2 } }{L_H} \cdot \epsilon_k \\
& = \frac{ - 5 + \sqrt{ 25 + 8 L_H \min\{ \| \nabla f(P(x_k+\alpha_k d_k) ) \mid_{J_k^-} \|/\epsilon_k^2 , 1\} } }{2L_H} \cdot \epsilon_k  \\
& = \frac{  4 \min\{ \| \nabla f(P(x_k+\alpha_k d_k) ) \mid_{J_k^-} \|/\epsilon_k^2 , 1\} }{5 + \sqrt{ 25 + 8 L_H \min\{ \| \nabla f(P(x_k+\alpha_k d_k) ) \mid_{J_k^-} \|/\epsilon_k^2 , 1\} } } \cdot \epsilon_k \\
& \ge \frac{4}{ 5 + \sqrt{25 + 8 L_H} } \min\{ \| \nabla f(P(x_k+\alpha_k d_k) ) \mid_{J_k^-} \| \epsilon_k^{-1}, \epsilon_k  \} \\
& \ge \frac{2}{\sqrt{25 + 8 L_H} } \min\{ \| \nabla f(P(x_k+\alpha_k d_k) ) \mid_{J_k^-} \| \epsilon_k^{-1}, \epsilon_k  \} \\ 
\notag
& \Downarrow (\alpha_k = 1, \| d_k \| = \| d^-_k \| ) \\
\notag
\alpha_k^2 \epsilon_k \| d_k \|^2 & \ge  \frac{4}{ 25 + 8 L_H} \min\{ \| \nabla f(P(x_k+\alpha_k d_k) ) \mid_{J_k^-}  \|^2 \epsilon_k^{-1}, \epsilon_k^3  \}.
\end{align*}
}

\noindent{\bf Case 2.} $ j_k = l_k \ge 1 $. 
In this case, since $\alpha_k = \theta^{m_k}$ with $m_k \le j_k = l_k$, we have 
\begin{align*}
     \theta^{ l_k } \| d_k \| > \theta \epsilon_k & \implies \alpha_k \|d_k \| = \theta^{m_k} \| d_k \| > \theta \epsilon_k \\
     & \implies  \alpha_k^2 \epsilon_k \| d_k \|^2 = (\alpha_k \| d_k \| )^2 \epsilon_k > \theta^2 \epsilon_k^3.
\end{align*}

\noindent{\bf Case 3.} $ j_k > l_k = 0 $. For $j = 0$ and $j = j_k - 1$, we must have
\begin{align}
\notag
& \theta^j  g_k^T d_k + \frac{ \theta^{2j} }{2}  d_k^T  H_k  d_k + \frac{L_H \theta^{3j} }{6} \| d_k \|^3 \ge - \eta \theta^{2j} \epsilon_k \| d_k \|^2 \\
\notag
\implies & \theta^j  (g^-_k)^T d^-_k + \frac{ \theta^{2j} }{2}  (d^-_k)^T  H^-_k d^-_k + \frac{L_H \theta^{3j} }{6} \| d^-_k \|^3 \ge - \eta \theta^{2j} \epsilon_k \| d^-_k \|^2 \\
\notag
\overset{ \eqref{ineq: CR} }{\implies} & - \frac{ \theta^j }{2} (1-\zeta) \epsilon_k \| d^-_k \|^2 + \frac{L_H}{6} \theta^{3j} \| d^-_k \|^3 \ge - \eta \theta^{2j} \epsilon_k \| d^-_k \|^2 \\
\label{ineq: pp7}
\implies & \frac{L_H}{6} \theta^{2j} + \frac{\eta \epsilon_k }{ \| d^-_k \| } \theta^j - \frac{(1-\zeta) \epsilon_k }{ 2 \| d^-_k \| } \ge 0.
\end{align}
By setting $j = 0$ in this inequality, we have $ \|  d^-_k \| \ge ( 3(1-\zeta) - 6\eta ) \epsilon_k / L_H $. 
By setting $j = j_k- 1$ in this same inequality, and using $\theta^{j_k} > \theta^{2j_k}$, we have
\begin{align}
\notag
\left( \frac{L_H}{6} + \frac{\eta \epsilon_k}{ \| d^-_k \| } \right) \theta^{j_k - 1} & \ge \frac{(1-\zeta) \epsilon_k}{2 \| d^-_k \|} \\
\label{ineq: pp8}
\implies
\theta^{j_k} \| d^-_k \| & \ge \frac{(1-\zeta) \theta \epsilon_k}{(L_H/3)+ 2\eta\epsilon_k/\| d^-_k \|} \\
\notag
& \ge \frac{(1-\zeta) \theta \epsilon_k}{(L_H/3)+ 2\eta L_H/(3(1-\zeta)-6\eta) } \\
\notag
& = \frac{ 3(1-\zeta-2\eta)\theta \epsilon_k }{L_H},
\end{align}
where the final equality follows by elementary manipulation.
Using again $\alpha_k = \theta^{m_k} \ge \theta^{j_k}$, we have
\[
\alpha_k^2 \epsilon_k \| d_k \|^2 = \alpha_k^2 \epsilon_k \| d^-_k \|^2 \ge ( \theta^{j_k} \| d^-_k \| )^2 \epsilon_k \ge \frac{9 (1-\zeta - 2\eta)^2 \theta^2 \epsilon_k^3}{L_H^2}.
\]

\noindent{\bf Case 4.} $ j_k > l_k \ge 1$. 
By the same argument as in Case 3, \eqref{ineq: pp8} holds. 
Moreover, $\| d^-_k \| = \| d_k \| > \epsilon_k$ since $l_k \ge 1$. 
Therefore, we have
\begin{align*}
\eqref{ineq: pp8} & \implies \theta^{j_k} \| d^-_k \| \ge \frac{(1-\zeta) \theta \epsilon_k}{L_H/3+ 2\eta\epsilon_k/\| d^-_k \|} > \frac{(1-\zeta) \theta \epsilon_k}{L_H/3+ 2\eta} \\
& \implies \alpha_k^2 \epsilon_k \| d_k \|^2 \ge ( \theta^{ j_k} \| d_k \|)^2 \epsilon_k \ge \frac{ (1-\zeta)^2\theta^2\epsilon_k^3 }{ (L_H/3 + 2\eta)^2 }
\end{align*}

\medskip

By combining the four cases analyzed above, we obtain
\[
\alpha_k^2 \epsilon_k \|d_k \|^2 \ge \frac{1}{\eta} \csol \min\{ \| \nabla f(P(x_k+\alpha_k d_k) ) \mid_{J_k^-}  \|^2 \epsilon_k^{-1}, \epsilon_k^3  \}.
\]
Therefore,  by the line search rule, \eqref{ineq: sd-sol} holds.\qed
\end{proof}

\begin{lemma} \label{lm: meostep}
Suppose that at iteration $k$ of Algorithm~\ref{Alg: PNCG}, Procedure~\ref{alg:meo} is invoked and identifies a direction with curvature less than or equal to $-\tfrac12 \epsilon_k$. Then we have 
\begin{align*}
f(x_k) - f(x_{k+1}) > \eta \min\left\{ \frac{ (3 - 6\eta)^2 \theta^2}{8 L_H^2} , \frac{\theta^2}{2} ,  \frac{1}{8} \right\} \epsilon_k^3 \ge \min \left\{ \frac{\cnc}{8}, \frac{\eta}{8} \right\} \eps_k^3.
\end{align*}
\end{lemma}
\begin{proof}
Let scalar $\lambda$ and vector $d$ be the quantities returned by MEO, Procedure~\ref{alg:meo}, so that $d^T S_k H_k S_k d = \lambda \le -\epsilon_k/2$ and $\| d \| = 1$. 
From the subsequent definition of $d_k$ in Algorithm~\ref{Alg: PNCG}, we have that
\begin{subequations}
\label{ineq: misc1}
\begin{align}
  g_k^T S_k d_k & = - | g_k^T S_k d | | d^T S_k H_k S_k d  | \le 0, \\
  \label{ineq: misc1a}
  \| d_k \| & =  | d^T S_k H_k S_k d  | \| d \| = | \lambda | \ge \tfrac12 \epsilon_k, \\
  \label{ineq: misc1c}
d_k S_k H_k S_k d_k & = ( d^T S_k H_k S_k d )^3 = \lambda^3 = -\| d_k \|^3.
\end{align}
\end{subequations}
Then, for any $0 < \gamma < \frac{3 - 6 \eta}{L_H}$, we have
\begin{align*}
\notag
f(x_k + \gamma S_k d_k ) & \le f(x_k) + \gamma g_k^T S_k d_k + \frac{\gamma^2}{2} d_k^T S_k H_k S_k d_k + \frac{L_H}{6} \gamma^3 \| S_k d_k \|^3 \\
& \overset{(S_k[i,i] \le 1)}{\le} f(x_k) + \gamma g_k^T S_k d_k + \frac{\gamma^2}{2} d_k^T S_k H_k S_k d_k + \frac{L_H}{6} \gamma^3 \| d_k \|^3 \\
\notag
& \overset{ \eqref{ineq: misc1} }{\le} f(x_k) - \frac{\gamma^2}{2} \| d_k \|^3 + \frac{L_H}{6} \gamma^3 \| d_k \|^3 \\
& < f(x_k) - \eta \gamma^2 \| d_k \|^3.
\end{align*}
Note that if $\gamma \| d_k \| \le \epsilon_k < 1$ then $\gamma \| d_k \|_\infty \le \epsilon_k < 1$ and $ P(x_k + \gamma S_k d_k) = x_k + \gamma S_k d_k$. 
In fact, by invoking \eqref{eq:Sk.def}, we have
\begin{align*}
    i \in J_k^+ \; & \implies \; x_k^i+\gamma s_k^i  d_k^i \ge x_k^i - x_k^i \| \gamma d_k \|_{\infty} \ge 0, \\
    i \in J_k^- \cap \sI \; & \implies \; x_k^i + \gamma s_k^i d_k^i = x_k^i + \gamma d_k^i \ge x_k^i - \epsilon_k >0.
\end{align*}
Thus for any $\gamma < \min\left\{ \frac{3 - 6\eta}{L_H}, \frac{\epsilon_k}{ \| d_k \| } \right\}$, we have  
\[
f(P(x_k + \gamma S_k d_k )) = f(x_k + \gamma S_k d_k) < f(x_k) - \eta \gamma^2 \| d_k \|^3.
\]
Therefore, because of the backtracking mechanism and the definition of $\bar {m}_k$, we have
\begin{align}
\notag
\theta^{\bar m_k} & \ge \min\left\{ \theta \min\left\{ \frac{3 - 6\eta}{L_H}, \frac{\epsilon_k}{ \| d_k \| } \right\}, 1 \right\} \\
\label{ineq: sslb}
\implies  \theta^{\bar m_k} \| d_k \| & \ge \min \left\{ \frac{(3-6\eta)\theta \| d_k \|}{L_H }, \theta \epsilon_k, \| d_k \| \right\} \\
\notag
& \overset{ ( \| d_k \| = | \lambda | \ge \frac{\epsilon_k}{2}  ) }{\ge} \min \left\{ \frac{(3-6\eta)\theta }{2 L_H }, \theta, \frac{1}{2} \right\} \epsilon_k.
\end{align}
Then, based on the line search rule and the bounds \eqref{ineq:
  sslb} and \eqref{ineq: misc1a}, we have
\begin{align*}
  f(x_k) - f(x_{k+1}) & = f(x_k) - f(P(x_k + \theta^{\bar m_k} S_k d_k)) \\
  & > \eta \theta^{2 \bar m_k} \| d_k \|^3 \\
& \ge \eta \min\left\{ \frac{ (3 - 6\eta)^2 \theta^2 }{8 L_H^2} , \frac{\theta^2}{2} ,  \frac{1}{8} \right\} \epsilon_k^3.
\end{align*}
The final inequality follows from the definition of $\cnc$ in Lemma~\ref{lm: ncdec}.\qed
\end{proof}

We now state and prove the main complexity result for
Algorithm~\ref{Alg: PNCG}. Note that $\epsilon_g$ is the parameter in the condition triggering the Newton-CG step in Algorithm~\ref{Alg: PNCG}.
%
\begin{theorem}\label{thm: PNCG.iter.comp} 
Suppose that Assumptions~\ref{Ass: comp.lev.set} and \ref{Ass: 2Lipstz} hold for the problem \eqref{opt: bc0}, \eqref{def:O}.
Consider Algorithm~\ref{Alg: PNCG} with $\epsilon_k \equiv \epsilon_H < 1$. 
\rev{Then Algorithm~\ref{Alg: PNCG} will stop within}
\begin{align} \label{def: Kpncg}
\Kpncg \triangleq \left\lfloor \frac{ 16 (f(x_0) - \fl) }{ \min\left\{ \cnc, 8 \csol, \frac{2\theta}{L_g},  \eta \right\} } \max \{ \epsilon_g^{-2} \epsilon_H, \epsilon_H^{-3}  \}  \right\rfloor + 2
\end{align}
iterations, and outputs a vector $x \in \Omega$ such that \rev{the following approximate first-order optimality conditions hold}
\begin{subequations} \label{e2o}
\begin{alignat}{2} \label{e2o:1}
& x^i \ge 0 \;\; \mbox{for $i \in \sI$}, \quad && \| S \nabla f(x) \| \le \epsilon_g + \epsilon_H^2, \\
\label{e2o:2}
& \nabla_i f(x)  \ge - \epsilon_H^{3/2}, \, && \forall i \in J^+ \triangleq \{  i \in \sI  \mid 0 \le x^i  \le \epsilon_H \},
\end{alignat}
\end{subequations}
\rev{with probability 1}. 
Moreover, $S \nabla^2 f(x) S \succeq - \epsilon_H I $ with probability at least $(1 - \delta)^{\Kpncg}$, 
where $S = \diag(s)$ is a diagonal matrix with $s^i = x^i, \forall i \in J^+$ and $s^i=1$ otherwise; and  $\delta \in [0,1)$ is the probability of failure in Procedure~\ref{alg:meo}. 
In particular, if we set $\epsilon_g = \epsilon$ and $\epsilon_H = \sqrt{\epsilon}$, then the algorithm outputs an $(\epsilon,1/2)$-2o point (according to Definition~\ref{def:em2o}) with  probability at least $(1-\delta)^{\Kpncg}$ within $\sO(\epsilon^{-3/2})$ iterations.
\end{theorem}
\begin{proof}
We prove by estimating the function decrease when the algorithm does not stop at iteration $k$ or $k+1$. \\
{\noindent \bf Case 1.} A gradient projection step is taken at iteration $k$. Then by Lemma~\ref{lm: gradproj}, we have
\begin{equation} \label{eq:case2}
f(x_k) - f(x_{k+1}) > \frac{1}{4} \min\left\{ \frac{\theta}{L_g}, 1 \right\} \epsilon_k^3.
\end{equation}
{\bf Case 2.} 
The Newton-CG step is triggered at iteration $k$, $J_{k+1}^- \neq \emptyset$ and $\| g^-_{k+1} \| > \epsilon_g$. Note that $\epsilon_k \equiv \epsilon_H$ indicates that $J_{k+1}^- \subseteq J_k^-$. Therefore, we have
\begin{align*}
\| \nabla f(x_{k+1}) \mid_{J_k^-} \| \ge \| g^-_{k+1} \| > \epsilon_g.
\end{align*}
Thus, by Lemma~\ref{lm: ncdec} and Lemma~\ref{lm: soldec}, we have that 
\begin{align*}
f(x_k) - f(x_{k+1})
& \ge \min\{ \cnc, \csol \} \min \{ \| \nabla f(x_{k+1}) \mid_{J_k^-} \|^2 \epsilon_H^{-1}, \epsilon_H^3  \} \\ 
& > \min\{ \cnc, \csol \} \min\{ \epsilon_g^2 \epsilon_H^{-1}, \epsilon_H^3 \}.
\end{align*}
{\noindent \bf Case 3.} The MEO procedure is triggered and a negative curvature step is taken at iteration $k$. Lemma~\ref{lm: meostep} then implies that
\begin{equation} \label{eq:case3}
f(x_k) - f(x_{k+1}) > \min\left\{ \frac{\cnc}{8},   \frac{\eta}{8} \right\} \epsilon_k^3.
\end{equation}
{\noindent \bf Case 4.} The Newton-CG step is triggered at iteration $k$, but $J_{k+1}^- = \emptyset$ or $\| g^-_{k+1} \| \le \epsilon_g$. We have from Lemmas~\ref{lm: ncdec} and \ref{lm: soldec} that $f(x_k) > f(x_{k+1})$. Moreover,  since the algorithm does not stop at iteration $k+1$,  $x_{k+2}$ is calculated from a step that is analyzed in either Case 1 or Case 3. It follows that either \eqref{eq:case2} or \eqref{eq:case3} is satisfied with $k$ replaced by $k+1$.

We now combine the lower bounds for function value decrease derived in the above four cases, let $\epsilon_k \equiv \epsilon_H < 1$, and we have that for any $k \ge 0$ such that the algorithm does not stop at iteration $k$ and $k+1$, that
\[
f(x_k) - f(x_{k+2}) > 
\min\left\{ \csol, \frac{\cnc}{8}, \frac{\theta}{4L_g}, \frac{\eta}{8} \right\} \min \{ \epsilon_g^2 \epsilon_H^{-1} ,\epsilon_H^3  \}
\]
if the stopping criterion is not satisfied. Therefore, the algorithm must stop within
the number of iterations stated in the theorem. When the algorithm stops, the output $x_k$ satisfies:
\begin{align}\label{outputcond.1}
\| g^-_k \| \le \epsilon_g, \quad g_k^i \ge -\epsilon_H^{3/2}, \ \forall i \in J_k^+,  \quad \| S^+_k g^+_k \| \le \epsilon_H^2.
\end{align} 

Now let us derive the probability that the output $x_k$ does not satisfies $S_k H_k S_k \succeq -\epsilon_H I$. Denote by $p_{k, F}$ the probability that the algorithm does not stop before iteration $k-1$ and $x_k$ does not satisfy $\lambda_{\min} (S_k \nabla^2 f(x_k) S_k) \ge -\epsilon_H$. 
(We set $p_{0,F} \triangleq 1$.) 
Denote by $p_{k,F,stop}$ the probability that the algorithm stops at iteration $k$ but $x_k$ does not satisfy $\lambda_{\min} (S_k \nabla^2 f(x_k) S_k) \ge -\epsilon_H$. 
Therefore, since the failure probability of Procedure~\ref{alg:meo} is $\delta$, we have that
\begin{align*}
p_{k,F,stop} \le \delta p_{k, F}.
\end{align*}
We know that the algorithm must stop within $\Kpncg$ number of iterations. Therefore, if we denote the probability of failure of PNCG as $p_F$, then
\begin{align*}
    p_F = \sum_{k = 0}^{\Kpncg - 1} p_{k,F,stop}.
\end{align*}
We have that for any $k = 0,1, \dotsc, \Kpncg-1$ that
\begin{align*}
    p_{k,F} + \sum_{t=0}^{k-1}p_{t,F,stop} \le 1,
\end{align*}
so that
\begin{align*}
p_{k,F,stop} \le \delta\left( 1 - \sum_{t=0}^{k-1} p_{t,F,stop} \right), \quad k = 0,1,\dotsc,\Kpncg-1.
\end{align*}
Next we show that $ \sum_{t=0}^k p_{t,F,stop} \le 1 - (1-\delta)^{k+1}$, $k = 0,1,\dotsc,\Kpncg-1 $ by induction. 
The claim is trivial for $k = 0$. 
Supposing that it holds when $k = \bar k \in \{0,1,\dotsc,\Kpncg-2 \}$, we have
\begin{align*}
    \sum_{t=0}^{\bar k + 1} p_{t,F,stop} & = \sum_{t=0}^{\bar k} p_{t,F,stop} + p_{\bar k + 1,F,stop} \\
    & \le \sum_{t=0}^{\bar k} p_{t,F,stop} + \delta\left( 1 - \sum_{t=0}^{\bar k} p_{t,F,stop} \right) \\
    & = \delta + (1-\delta) \sum_{t=0}^{\bar k} p_{t,F,stop} \\
    & \le \delta + (1-\delta) [ 1 - (1-\delta)^{\bar k + 1} ] \\
    & = 1 - (1-\delta)^{\bar k + 2}.
\end{align*}
This proves that the desired bound holds for $k = \bar k + 1$, completing the induction. 
Therefore, we have that
\begin{align*}
    p_F = \sum_{k=0}^{\Kpncg-1} p_{k,F,stop} \le 1 - (1-\delta)^{\Kpncg}.
\end{align*}
Then we proved that with probability at least $(1-\delta)^{\Kpncg}$, the output $x_k$ satisfies $S_k H_k S_k \succeq - \epsilon_H I $.
This condition for $x_k$ combined with \eqref{outputcond.1} indicate the output property.\qed
\end{proof}


\rev{In the statement of Theorem~\ref{thm: PNCG.iter.comp}, $\delta$ is a user-defined parameter. It can be chosen small enough to ensure that $(1-\delta)^{\Kpncg}$ is large. Specifically, by Bernoulli's inequality, for $\delta \in [0,1)$ and $K \ge 1$,
$$ (1-\delta)^K \ge 1 - K \delta. $$
If, for example, we set $\delta = 0.01/\Kpncg$, then $(1-\delta)^{\Kpncg} \ge 1 - 0.01 = 0.99$. Note that the value of $\delta$ only affects the operation complexity (involving Hessian-vectors products), which depends only {\em logarithmically}  on $\delta$ (see Corollary~\ref{cor:oper.comp} below).
Therefore, we are free to choose very small values of $\delta$ without affecting the operation complexity significantly.}

We now state a result for operation complexity of this approach, based on the fundamental operations of gradient evaluation and Hessian-vector products.
\begin{corollary}\label{cor:oper.comp}
\rev{Suppose that Assumptions~\ref{Ass: comp.lev.set}, \ref{Ass: 2Lipstz} hold for the problem \eqref{opt: bc0}, \eqref{def:O}.
For some $\epsilon \in (0,1)$, consider Algorithm~\ref{Alg: PNCG} with $\epsilon_k \equiv \sqrt{\epsilon}$ and $\epsilon_g = \epsilon$. 
Then Algorithm~\ref{Alg: PNCG} stops and outputs an $(\epsilon,1/2)$-2o point with probability at least $(1-\delta)^{\Kpncg}$ ($\Kpncg$ defined in \eqref{def: Kpncg}) within
\begin{align*}
    O\left(\epsilon^{-3/2}  \min \left\{ n, \epsilon^{-1/4}  \log\left( \frac{n}{\delta \epsilon} \right)  \right\} \right).
\end{align*}
fundamental operations (gradient evaluations or Hessian-vector products).}
\end{corollary}
\rev{
\begin{proof}
The bound on Hessian-vector products before Algorithm~\ref{Alg: PNCG} stops is:
\begin{align}\label{hvbd}
    \sum_{k=0}^{\Kpncg - 1}(\max\{ 2 \min \{ n, \mathbb{J}_k \} + 1, \Nmeo_k \}),
\end{align}
where $2 \min \{n, \mathbb{J}_k \} + 1$ and $\Nmeo_k$ are the bound on Hessian-vector products of the Capped CG and MEO procedure, respectively, at iteration 
$k$. By Lemma~\ref{lm:CGoper.comp.} and \ref{lm:MEOoper.comp.} in Appendix~\ref{app: capped-CG} and \ref{app: MEO}, given $\kappa \triangleq \frac{\| H_k^- \| + \epsilon_k}{\epsilon_k} \le \frac{ L_g + \epsilon_k}{\epsilon_k}$, $\Cmeo_k = \log \left( \frac{2.75n}{\delta^2}\right) \frac{\sqrt{ \| H \| }}{2} \le \log \left( 2.75n/\delta^2 \right) \sqrt{ L_g }/2$ and $\epsilon_k \equiv \sqrt{\epsilon}$, we have that:
\begin{align*}
    \mathbb{J}_k & \le \min\left\{n, \left\lceil \left( \sqrt{\kappa} + \frac{1}{2} \right) \log \left( \frac{144 (\sqrt{\kappa} + 1)^2 \kappa^6 }{\zeta^2} \right) \right\rceil \right\} \\
    \implies \mathbb{J}_k & = \sO\left(\min\left\{n,\epsilon^{-\frac{1}{4}}\log\left(\epsilon^{-1} \right)\right\}\right) \\
    \Nmeo_k & = \min \left\{ n, 1 + \lceil \Cmeo_k \epsilon_k^{-\frac{1}{2}} \rceil \right\} = \sO\left( \min\left\{n,\epsilon^{-\frac{1}{4}}\log\left(n/\delta \right) \right\} \right),
\end{align*}
Therefore, by Theorem~\ref{thm: PNCG.iter.comp} we have that
\begin{align*}
    \eqref{hvbd} & \le 
    \sum_{k=0}^{\Kpncg - 1} 2 (\max\{ \mathbb{J}_k, \Nmeo_k \} + 1) \\
    & = \sO \left(\Kpncg \min \left\{ n, \epsilon^{-\frac{1}{4}} \max \left\{ \log\left(\epsilon^{-1} \right), \log\left(n/\delta \right) \right\} \right\} \right) \\
    & = \sO \left(\Kpncg \min \left\{ n, \epsilon^{-\frac{1}{4}} \left( \log\left(\epsilon^{-1} \right) + \log\left(n/\delta \right)  \right) \right\} \right) \\
    & = \sO \left(\epsilon^{-\frac{3}{2}}  \min \left\{ n, \epsilon^{-\frac{1}{4}}  \log\left( \frac{n}{\delta \epsilon} \right)  \right\} \right)
\end{align*}
Then the result follows by noticing that the number of gradient evaluation is bounded by the number of outer-loop iterations of Algorithm~\ref{Alg: PNCG}, i.e., $\Kpncg$. \qed
\end{proof}}


\section{Numerical experiment}\label{sec: num}
We test the practicality of {\bf PNCG} (Algorithm~\ref{Alg: PNCG}) by comparing it with several other approaches on the well-known Nonnegative Matrix Factorization (NMF) problem.
The competitors include the gradient projection method ({\bf pgrad}) described in \cite[Section 3.3]{bertsekas2016nonlinear} (see Algorithm~\ref{Alg: pgrad}),
a log-barrier Newton-CG ({\bf LBNCG}) proposed in \cite{10.1093/imanum/drz074} for  optimization with bounds, and two approaches that are specialized to NMF.
Preliminary results show that {\bf PNCG} contends well with {\bf pgrad} and {\bf LBNCG}, and is competitive with the specialized methods on problems with relatively low dimensions.\footnote{Experiments in this section are conducted using \texttt{Matlab R2018b} on MacBook Air 1.3 GHz Intel Core i5. Source codes of experiments in this section can be found at: \texttt{https://github.com/yue-xie/ProjectedNewton}.} 
We use $\left<A,B\right>$ to denote the  inner product of matrices $A, B \in \bR^{d_1 \times d_2}$ defined by  $Tr(A^TB)$, while the Frobenius norm  is $\| A \|_F = \sqrt{\left<A,A\right>}$.

NMF is stated as follows, for a given matrix $V \in \bR^{m \times n}$:
\begin{equation} \label{NMF}
\min_{W \in \bR^{m \times r}, \Y \in\bR^{r \times n}} \, F(W,\Y) \triangleq \frac{1}{2} \| W \Y - V \|_F^2, \quad \mbox{subject to  $W \ge 0$, $ \Y \ge 0$},
 \end{equation}
 where the  nonnegativity constraints apply componentwise, that is, all elements of $W$ and $\Y$ are required to be nonnegative.
NMF has a wide range of applications in image processing and text mining; see  \cite{gillis2014and} for a comprehensive review.

In all following experiments, we create synthetic datasets following the approach in \cite{kim2008toward}: Matrices $\bar{W} \in \bR^{m \times r}$ and $\bar{\Y} \in \bR^{r \times n}$ are generated randomly where each element has half standard normal distribution (to ensure $\bar{W} \ge 0$ and $\bar{\Y} \ge 0$). 
Then approximately $60\%$ of the elements of these matrices (chosen uniformly at random) are replaced by zeros. 
We then set  $V = \bar{W} \bar{\Y} + E$,  where $E$ is elementwise Gaussian  with  mean $0$ and standard deviation of $5\%$ of average elementwise magnitude of $\bar{W} \bar{\Y}$. Finally, $V$ is normalized such that its average elementwise magnitude is $1$.

\subsection{Comparison with other solvers with complexity guarantees}
\label{subsec: exp1} 
In this subsection we solve NMF using {\bf PNCG} and other solvers, including the gradient projection method ({\bf pgrad}) and the log-barrier Newton-CG ({\bf LBNCG}). The former is a known practical method for constrained nonlinear optimization \cite[Section 3.3]{bertsekas2016nonlinear}. However, it is only guaranteed to seek an approximate first-order optimal point; its complexity guarantees ($\sO(\epsilon^{-2})$) (c.f. \cite{ghadimi2016mini}) are generally worse than second-order methods ($\sO(\epsilon^{-3/2})$) in the nonconvex regime. The latter is proposed in \cite{10.1093/imanum/drz074}, which does have competitive complexity guarantees (see Table~\ref{tab: complex.}). Although {\bf PNCG} and {\bf LBNCG} are able to locate approximate second-order optimal solutions, we stop these algorithms as long as a first-order point is found or time/iteration limit is reached, so that comparison with {\bf pgrad} is fair.

\paragraph{Methods.} 
First we specify the methods implemented in the experiment and their settings. 
We make use here of notation $\nabla^P$ introduced in \cite{lin2007projected} and  defined as follows:
\begin{align} \label{def: projgrad}
  \begin{aligned}
  \nabla^P_i f(x) = \begin{cases}
    \nabla_i f(x) & \ \mbox{if $x^i > 0$ or $i \in \sI^c$},\\
    \min\{0, \nabla_i f(x) \} & \ \mbox{if $x^i = 0$ and $i \in \sI$}.
    \end{cases}
\end{aligned}
\end{align}
(Note that $\nabla^P f(x) = 0$ implies the first-order optimality conditions of \eqref{1o-stat}.)

\be
\item[1.] {\bf PNCG} (Algorithm~\ref{Alg: PNCG})\footnote{Note that Assumption~\ref{Ass: comp.lev.set} may not hold for \eqref{NMF}, but we can modify the formulation to ensure this property, for example by adding elementwise upper bounds to $W$ and $H$ or adding a penalty $\| W^TW - \Y \Y^T \|^2$ to the objective. We omit these modifications to allow a more direct comparison with the specialized solvers for \eqref{NMF} described later.}: Set $\epsilon_g = 10^{-6}$, $\epsilon_k \equiv \sqrt{\epsilon_g}$, $\theta = \zeta = 1/2$, $\eta = 0.2$. For the parameter $\hat \zeta$ in Algorithm~\ref{alg:ccg}, we set it initially $.1$, but decrease by a factor of $10$ whenever the line search procedure in the outer-loop fails to find a descent direction, until a lower bound of $\frac{\zeta}{3 \kappa}$ is reached. 
We do not use Procedure MEO, terminating Algorithm~\ref{Alg: PNCG} when $g_k^i \ge -\epsilon_k^{3/2}$ for all $i \in J_k^+$ and $\| S^+_k g^+_k \| \le \epsilon_k^2$ and $ \| g^-_k \| \le \epsilon_g $, because we are interested only in finding an approximate first-order solution satisfying \eqref{e2o}.
\item[2.] {\bf pgrad} (Algorithm~\ref{Alg: pgrad}): Projected gradient method \cite[Section 3.3]{bertsekas2016nonlinear} directly applied to NMF. This method uses Armijo rule along the projection arc \cite{bertsekas2016nonlinear}, with backtracking parameter $\beta = 1/2$ and step acceptance parameter $\sigma = 1/2$ is chosen as such to be consistent with the gradient projection step in Algorithm~\ref{Alg: PNCG} (where the step acceptance parameter is set as the default value $1/2$).
This algorithm is terminated when $\| \nabla^P F(W,\Y) \|_F \le 10^{-4}$. 
\begin{algorithm}[ht!]
\caption{Projected gradient method for NMF ({\bf pgrad})}\label{Alg: pgrad}
\begin{algorithmic}
\STATE {\bf (Initialization)} Choose initial nonnegative real matrices $W_0, \Y_0$, backtracking parameter $\beta \in (0,1)$ and step acceptance parameter $\sigma \in (0,1)$. 
\FOR{$k=0,1,2,\dotsc$}
   \STATE Let $m_k$ be the smallest nonnegative integer $m$ such that
  \begin{align*}
  & F( W_k, \Y_k ) - F( W_k( \beta^m ), \Y_k(\beta^m) ) \\
  & > \sigma ( \left\langle W_k-W_k(\beta^m), \nabla_W F(W_k,\Y_k) \right\rangle +\left\langle \Y_k-\Y_k(\beta^m), \nabla_Y F(W_k,\Y_k) \right\rangle ),
  \end{align*}
  where $ W_k(\alpha) \triangleq \max(W_k - \alpha \nabla_W F(W_k,\Y_k),0)$ \\
  \;\; and $\Y_k(\alpha) \triangleq \max(\Y_k - \alpha \nabla_Y F(W_k,\Y_k),0)$.
\STATE Let $ W_{k+1} := W_k (\beta^{m_k} ) $; $ \Y_{k+1} := \Y_k (\beta^{m_k} ) $.
\ENDFOR
\end{algorithmic}
\end{algorithm}
\item[3.] {\bf LBNCG}: Log-barrier Newton-conjugate-gradient \cite{10.1093/imanum/drz074}. 
This method is \\equipped with worst case complexity guarantees (see Table~\ref{tab: complex.}) but its practical performance has not been studied to date. We implement it as Algorithm~1 in \cite{10.1093/imanum/drz074} with parameter choices $\epsilon_g = 10^{-4}$, $\theta = 1/2$, $\xi_r = 1/2$, $\bar \xi = 1/2$, $\beta = 1/2$, $\eta = 1/2$. We deal with the CG accuracy tolerance $\hat \xi_r$ and $c_{\mu}$ 
similarly as in our implementation of {\bf PNCG}, setting  them initially to $.1$ and decreasing them when we find that the modified CG is not yielding descent directions. Similar as in PNCG, we turn off Procedure 3 (MEO) in Algorithm 1 in \cite{10.1093/imanum/drz074} because we are only interested in locating an approximate first-order solution. Termination criterion is $\nabla F(W_k,\Y_k) > -10^{-4}$ and $ | \min\{[W_k^T,Y_k],1\}\odot[\nabla_W F(W_k,Y_k)^T, \nabla_Y F(W_k,Y_k)] | \le 10^{-4} $, where $>, \le, \min\{\}, |\cdot|$ hold elementwisely and $\odot$ denotes elementwise multiplication.
\ee

An outer-loop iteration limit 
of 5000 and a running time of 100s are set for {\bf PNCG} and {\bf pgrad}. 
An outer-loop iteration limit of 10000 and a time limit of 60s are applied to   {\bf LBNCG}. 

\paragraph{Experiment settings and metrics.} To create Table~\ref{tab: comparison on NMF}, we generate three different scenarios ($(m,n,r) \in \{ (150,100,15) , (300,200,15), (600,400,15) \}$). The elements of the initial matrices $W_0$ and $\Y_0$ are chosen from the half standard normal distribution, then normalized so that the average elementwise magnitude of either $W_0$ or $Y_0$ is $1$. 
Given $\bar x \ge 0$, the residual of \eqref{opt: bc0},\eqref{def:O} is defined following Definition~\ref{def:em2o}: 
\begin{align}\label{def: res.}
\mbox{residual} = \max\left\{ \| \bar S \nabla f(\bar x) \|, -\min_{i \in J^+} \left\{ \nabla_i f(\bar x) \right\} \right\}, \end{align}
where $J^+ \triangleq \{ i \in \sI \mid 0 \le \bar{x}^i \le \sqrt{\epsilon_r} \}$, $J^- \triangleq \{1,\hdots,n\} \setminus J^+ =  \sI^c \cup \{ i \in \sI \mid \bar{x}^i > \sqrt{\epsilon_r} \}$, and $\bar S = \diag(\bar{s})$ is a diagonal matrix with $\bar{s}^i= 1$ when $i \in J^-$ and $\bar{s}^i =\bar{x}^i$ when $i \in J^+$. 
In this experiment we let $\epsilon_r = 10^{-6}$. 

\paragraph{Results.} Table~\ref{tab: comparison on NMF} indicates that {\bf PNCG} and {\bf pgrad} are close in performance, with {\bf PNCG} attaining slightly better residual measures. 
{\bf PNCG} requires fewer outer-loop iterations because the Newton-CG steps taken on some iterations yield more progress than a first-order step. 
{\bf LBNCG} is not competitive, perhaps not surprisingly since this method was designed with good {\bf worst-case} complexity in mind, rather than for any practical considerations.

\begin{table}
  \caption{Comparison of three solvers with complexity guarantees on NMF. Three scenarios are considered with different dimensions $m$ and $n$. In each scenario, 5 trials are run from different initial points (each of the four algorithms starts from the same initial point on each trial) and  average results are reported. Elapsed time of each algorithm is reported. $F^*$ is the objective function value of the output. \textbf{residual} is defined in \eqref{def: res.} and \textbf{projnorm} represents $\| \nabla^P F(W,\Y) \|_F$. {\bf PNCG} and {\bf pgrad} have similar performance that clearly dominates {\bf LBNCG}, which always fails to converge in the allotted time / iteration limit.
  }
  \label{tab: comparison on NMF}
  \centering
\begin{tabular}{cccccc}
    \toprule
  Algorithm & outer-loop iteration &  time(s) & $F^*$ & residual & projnorm  \\
  \midrule
  \multicolumn{6}{c}{$m = 150$, $n = 100$, $r = 15$} \\
  \midrule
  {\bf PNCG} & 1030.4 & 1.3 & 15.8 & 2.7e-05 & 6.3e-05 \\ 
{\bf pgrad} & 1275.2 & 1.4 & 15.8 & 9.4e-05 & 9.5e-05 \\ 
{\bf LBNCG} & 9774.8 & 57.7 & 4574.0 & 2.6e+04 & 4.4e+04 \\
\midrule 
\multicolumn{6}{c}{$m = 300$, $n = 200$, $r = 15$}\\
\midrule 
{\bf PNCG} & 639.4 & 2.3 & 68.9 & 2.8e-05 & 1.4e-04 \\ 
{\bf pgrad} & 708.8 & 2.2 & 68.9 & 9.3e-05 & 9.5e-05 \\ 
{\bf LBNCG} & 4529.4 & 60.0 & 23702.8 & 7.5e+04 & 1.4e+05 \\ 
\midrule
\multicolumn{6}{c}{$m = 600$, $n = 400$, $r = 15$}\\
\midrule 
{\bf PNCG} & 579.2 & 8.1 & 285.5 & 3.0e-05 & 1.1e-04 \\ 
{\bf pgrad} & 619.4 & 8.8 & 285.5 & 9.3e-05 & 9.7e-05 \\ 
{\bf LBNCG} & 1364.8 & 60.0 & 146213.1 & 2.1e+05 & 4.3e+05 \\ 
\bottomrule
  \end{tabular}
\end{table}

\subsection{Comparison with specialized NMF schemes}\label{subsec: exp2}

We now compare {\bf PNCG} with efficient alternating-direction schemes that are specialized for NMF.
The following methods are compared.
\be
\item[1.] {\bf PNCG}(Algorithm~\ref{Alg: PNCG}): We use the same settings as in in Section~\ref{subsec: exp1}, except that the MEO Procedure (Procedure \eqref{alg:meo}) is turned on and implemented using CG (see \cite[Theorem~1]{Royer2019}) with $\delta = .01$. This procedure enables PNCG to escape from a saddle point, as is shown in Figure~\ref{fig: NMF}(d). 
\item[2.] {\bf alspgrad}: Alternating nonnegative least squares  using projected gradient, described in \cite{lin2007projected}. Parameter settings are as described in \cite{lin2007projected}, except that the algorithm is stopped when $\| \nabla^P F(W,\Y) \|_F \le 10^{-4}$ (instead of $\| \nabla^P F(W,\Y) \|_F$ $\le 10^{-4} \times \| \nabla F(W_0,\Y_0) \|_F$ ) and the initial tolerance for the subproblem is set as $10^{-3}$ (instead of $10^{-3} \times \| \nabla F(W_0,\Y_0) \|_F$). 
\item[3.] {\bf pnm}: Alternating nonnegative least squares using two-metric gradient projection, described in \cite{gong2012efficient}. Parameter settings from \cite{gong2012efficient} are used, except that the algorithm is stopped when $\| \nabla^P F(W,\Y) \|_F \le 10^{-4}$ and the initial tolerance for the subproblem is set as $10^{-3}$.
\ee
An outer-loop iteration limit of 5000 and a running time limit of 100s are set for {\bf PNCG}, while limits of  1000 and 100s are applied to the other two algorithms.

\paragraph{Settings.} \revn{Synthetic datasets are created as above, with $m = 300$ and $n = 200$, and $r= 10,15$. 
For Table~\ref{tab: comparison with specialized solvers}, we use 10 cases of randomly generated datasets and initial matrices for each triple $(m,n,r)$ and record the average outcome. In particular, the initial matrices $W_0,\Y_0$ are generated i.i.d. elementwise from a half standard normal distribution, then normalized such that the average magnitude of either $W_0$ or $\Y_0$ is $1$. 
For Figure~\ref{fig: NMF}, we start both algorithms near a saddle point of \eqref{NMF}, constructed according to the following observation. 
If $U \in \bR^{m \times r_0}$ and $R \in \bR^{r_0 \times n}$ constitute a first-order optimal point of \eqref{NMF} (that is, $\nabla^P F(U,R) = 0$) when $r = r_0$, then 
\begin{align}\label{def: saddlept.}
W \triangleq \frac{1}{k_1} \underbrace{(1,\hdots,1)}_{k_1 \times k_2} \otimes U, \quad \Y \triangleq \frac{1}{k_2} \underbrace{(1;\hdots;1)}_{k_1 \times k_2} \otimes R
\end{align}
constitute a first-order optimal point of \eqref{NMF} when $r = r_0 k_1 k_2$. 
In the experiment, we first use {\bf alspgrad} to solve \eqref{NMF} with $r = r_0$ and obtain the approximate solution $U$ and $R$. We then set $W_0$ and $\Y_0$ as in \eqref{def: saddlept.}, and run {\bf alspgrad} with $r=r_0k_1k_2$ from this starting point, to see if it is able to escape from the saddle point. 
The other approaches are run from the same choice of $W_0$ and $\Y_0$. In Figure~\ref{fig: NMF}, we record three cases: $(r,r_0,k_1,k_2) = (10,1,5,2)$,
$(10,2,5,1)$,$(15,5,3,1)$.}

\paragraph{Results.} \revn{Table~\ref{tab: comparison with specialized solvers} averages results over ten runs for each choice of $(m,n,r)$. We see that {\bf PNCG} is slower in computation time (though within a factor of two of the fastest  specialized solver) but locates a slightly more accurate solution, as measured by the Frobenius norm of the projected gradient. Note that {\bf alspgrad} and {\bf pnm} are methods designed exclusively to solve NMF; they are not equipped with complexity results.}
When $m$ and $n$ are larger than the values used here, the discrepancy of computation time may become larger. In fact, the cost of the Hessian-vector product or gradient evaluation or checking step acceptance criterion in {\bf PNCG} is $O(mnr)$; the cost of the gradient evaluation or step acceptance criterion validation in the subproblem of {\bf alspgrad} is either $\sO(mr^2)$ or $\sO(nr^2)$; the cost of gradient evaluation or partial Hessian evaluation or step acceptance criterion validation in {\bf pnm} is either $\sO(mr^2)$ or $\sO(nr^2)$, while the step direction calculation in {\bf pnm} costs either $\sO(\bar m \bar s^3)$ or $\sO(\bar n \bar s^3)$ where $\bar m \le m$, $\bar n \le n$, $\bar s \le r$. Therefore, when $m \gg r$, $n \gg r$, the higher costs of these basic operations compromise the performance of PNCG. 

In Figure~\ref{fig: NMF}, where the algorithms are initialized near the saddle point, {\bf pnm} cannot be applied since the Hessian for the subproblem is singular at the initial point. The first-order method {\bf alspgrad} is able to escape from the vicinity of the current saddle point and reduce the objective further, but it appears to get stuck at another suboptimal point. Meanwhile, {\bf PNCG} appears to exit the saddle point, due to a call to the  MEO procedure, Algorithm~\ref{alg:meo}. We include Figure~\ref{fig: NMF} to verify the theory for {\bf PNCG} in the worst-case scenario of starting at a saddle point. Random starts like those used in the other plots are likely to yield convergence of the specialized methods to local minima.



\begin{table}
  \caption{Comparison between {\bf PNCG} and two solvers that are specialized to NMF, showing time, objective function value $F^*$ and \textbf{projnorm} $\| \nabla^P F(W,\Y) \|_F$ of the output. For each group of $(m,n,r)$, we generate synthetic data and initial points randomly. Then repeat for 10 times and report the average values of run time, final objective value, and norm of projected gradient.}
  \label{tab: comparison with specialized solvers}
  \centering
\begin{tabular}{cccc}
    \toprule
  Algorithm &  time(s) & $F^*$ & projnorm  \\
  \midrule
  \multicolumn{4}{c}{$m = 300$, $n = 200$, $r = 10$} \\
  \midrule
  {\bf PNCG} & 1.06 & 70.065 & 3.1e-05 \\ 
{\bf alspgrad} & 0.59 & 70.065 & 5.6e-05 \\ 
{\bf pnm} & 0.79 & 70.065 & 7.8e-05 \\
\midrule
  \multicolumn{4}{c}{$m = 300$, $n = 200$, $r = 15$} \\
  \midrule
  {\bf PNCG} & 2.37 & 68.692 & 6.9e-05 \\ 
{\bf alspgrad} & 1.70 & 68.692 & 8.1e-05 \\ 
{\bf pnm} & 1.19 & 68.692 & 8.0e-05 \\
\bottomrule
  \end{tabular}
\end{table}

\begin{figure}
    \centering
    \includegraphics[width = 1\textwidth]{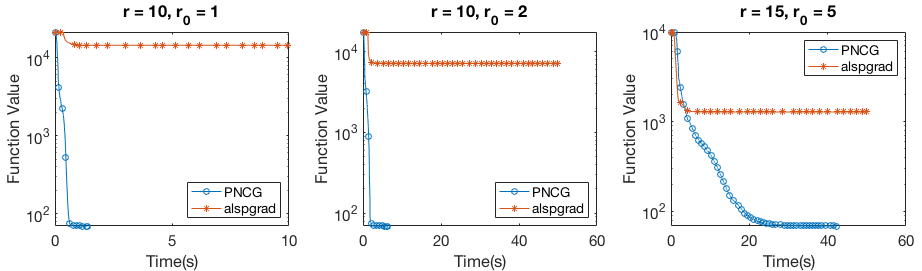}
    \caption{Comparison between {\bf PNCG} and specialized solver {\bf alspgrad}, showing the objective function value plotted against time. The algorithms are started near saddle points constructed as in \eqref{def: saddlept.}.
    }
    \label{fig: NMF}
\end{figure}

\section{Conclusion}\label{sec: con}
In this article, we relate and compare different definitions of approximate second-order optimal point in literature and define our own for optimization with bounds. We proposed a projected Newton-CG method. It has good complexity guarantees and is designed with practicality in mind, and is related to the two-metric projection algorithms proposed in the 1980s. The projected Newton-CG terminates within $\sO(\epsilon^{-3/2})$  iterations or $\tilde{\sO}(\epsilon^{-7/4})$ number of Hessian-vector product/gradient evaluation operations and finds a point that is approximately second-order optimal to tolerance $\epsilon$, with high probability. 
Numerical experiments on nonnegative matrix factorization illustrate practicality of the methods. 

\revn{In future work, we will consider extensions of the algorithms to solve optimization problems with more complex structures such as polyhedral feasible sets, $\ell_1$-norm terms, and even more general convex constraints that allow for cheap projection.} We will also investigate complexity guarantees of the two-metric projection algorithm proposed by Bertsekas.


\bibliographystyle{spmpsci}      
\bibliography{ref}

\appendix

\section{\rev{Comparing approximate second-order optimality conditions}}\label{app:2o}

We show that the relation $\gtrsim$ is transitive.
\begin{lemma}\label{def:trans}
If $\sodAA \gtrsim \sodBB$ and $\sodBB \gtrsim \sodCC$, then $\sodAA \gtrsim \sodCC$.
\end{lemma}
\begin{proof}
Since  $\sodAA \gtrsim \sodBB$, there exists $\epsilon_A \in (0,1]$ and $c_A > 0$ such that for any $x \in \sX$ and $\epsilon \in (0,\epsilon_A]$, if $x$ is an $(\epsilon,p)$-2o point by $\sodAA$, then $x$ is also a $(c_A \epsilon,p)$-2o point by $\sodBB$. 

Likewise, since $\sodBB \gtrsim \sodCC$, there exists $\epsilon_B \in (0,1]$ and $c_B > 0$ such that for any $x \in \sX$ and $\epsilon \in (0,\epsilon_B]$, if $x$ is an $(\epsilon,p)$-2o point by $\sodBB$, then $x$ is also a $(c_B \epsilon,p)$-2o point by $\sodCC$.

Let $\bar \epsilon = \min\{ \epsilon_A, \epsilon_B/c_A \}$. Take arbitrary $x \in \sX$ and $\epsilon \in (0,\bar \epsilon]$. Suppose that $x$ is an $(\epsilon,p)$-2o point by $\sodAA$. Since $\sodAA \gtrsim \sodBB$ and $\epsilon \le \epsilon_A$, $x$ is a $(c_A \epsilon,p)$-2o point by $\sodBB$; Since $\sodBB \gtrsim \sodCC$ and $c_A \epsilon \le \epsilon_B$, $x$ is also a $(c_B c_A \epsilon,p)$-2o point by $\sodCC$. Since the choice of $x \in \sX$ and $\epsilon \in (0,\bar \epsilon]$ is arbitrary, we have that $\sodAA \gtrsim \sodCC$. \qed
\end{proof}

We now give the proof of Theorem~\ref{thm: defcomp}. 
\begin{proof}
Let $U_{\sX} \triangleq \max_{x \in \sX} \| x \|_\infty$.
\be
\item[(1)] Suppose that $x$ is an $(\epsilon,p)$-2o point by $\sodb$ and $\epsilon \le  1$. We show (1) through the following steps.
\be
\item[(1a)] Fix any index $i$. Choose $d$ such that $d^i = \Delta$ and $d^j = 0, \forall j \neq i$, then $ \eqref{eq:em2o-cartis} \implies \nabla f(x)^T d = \nabla_i f(x) \Delta \ge - \Delta \epsilon \implies \nabla_i f(x) \ge -\epsilon $. This indicates that $\nabla f(x) \ge -\epsilon {\bf 1}.$
\item[(1b)] Fix any index $i$. Let $d^i = -\mbox{sign}(\nabla_i f(x)) \min\{\Delta,x^i\} $ and $d^j = 0$, $\forall j \neq i$. Then $ \eqref{eq:em2o-cartis} \implies \nabla f(x)^T d = - | \nabla_i f(x) | \min\{\Delta,x^i\} \ge - \Delta \epsilon \implies | \nabla_i f(x) | \le \frac{ \Delta \epsilon }{ \min\{\Delta,x^i\} } $ $ \implies |\nabla_i f(x) x^i | \le \frac{ \Delta x^i \epsilon}{\min\{\Delta,x^i\}} = \max\{ x^i, \Delta \} \epsilon \le \max\{ U_{\sX}, \Delta \} \epsilon.$ This indicates that
$\| X \nabla f(x) \|_\infty$ $\le \max\{ U_{\sX}, \Delta \} \epsilon \le \max\{ U_{\sX}, \Delta_{\max} \} \epsilon.$
\item[(1c)] If $x = 0$, then second row of \eqref{eq:em2o-Haeser} holds trivially. Suppose that $x \neq 0$. Let $d \triangleq c_d X v$, where $c_d \triangleq \min \left\{ \frac{\Delta}{\| x \|_\infty}, 1 \right\}$, $v$ is an arbitrary vector with $\| v \|_2 = 1$. Therefore, we have that $x + d \ge 0, x - d \ge 0, \| d \| \le \Delta$. Therefore,
\begin{align*}
    -\Delta^2 \epsilon^p & \overset{\eqref{eq:em2o-cartis}}{\le} \nabla f(x)^T d + \frac{1}{2} d^T \nabla^2 f(x) d
\end{align*}
and also,
\begin{align*}
    -\Delta^2 \epsilon^p & \overset{\eqref{eq:em2o-cartis}}{\le} -\nabla f(x)^T d + \frac{1}{2} d^T \nabla^2 f(x) d.
\end{align*}
Therefore,
\begin{align*}
    -\Delta^2 \epsilon^p & \le \frac{1}{2} d^T \nabla^2 f(x) d = \frac{c_d^2}{2} v^T X \nabla^2 f(x) X v  \\
    \implies v^T X \nabla^2 f(x) X v & \ge - \frac{2 \Delta^2}{c_d^2} \epsilon^p \ge - 2 \max\{ \| x \|_\infty^2, \Delta^2 \} \epsilon^p \ge - 2 \max\{ U_{\sX}^2, \Delta_{\max}^2 \} \epsilon^p.
\end{align*}
\ee
Denote $c_{\Delta} \triangleq \left( 2 \max\{ U_{\sX}^2, \Delta_{\max}^2 \} \right)^{1/p}$. Therefore, by (1a)-(1c), $x$ is an $(c\epsilon,p)$-2o point by $\sodc$, where $c \triangleq \max\{1, U_{\sX}, \Delta_{\max}, c_{\Delta} \}$.
\item[(2)] Given $x \ge 0$, let $T \triangleq \diag(t)$ be a diagonal matrix of $n\times n$ such that $t^i = 1$ if $x^i \le 1$ and $t^i = 1/x^i$ if $x^i > 1$. Then we have that $\bar X = XT = TX$.\\
$\sodc \gtrsim \sodd$. Suppose that $x$ is an $(\epsilon,p)$-2o point by $\sodc$. Note
\begin{align*}
    \| \bar X \nabla f(x) \|_\infty \le \| X \nabla f(x) \|_\infty \le \epsilon. \\
    d^T \bar X \nabla^2 f(x) \bar X d = (Td)^T X \nabla^2 f(x) X Td \ge -\epsilon^p \| Td \|^2 \ge -\epsilon^p \| d \|^2, \forall d \in \bR^n.
\end{align*}
Therefore, $x$ is also an $(\epsilon,p)$-2o point by $\sodd$.\\
$\sodd \gtrsim \sodc$. $x$ is an $(\epsilon,p)$-2o point by $\sodd$. Then
\begin{align*}
    \| X \nabla f(x) \|_\infty = \| T^{-1} \bar X  \nabla f(x) \|_\infty \le \| x\|_\infty \| \bar X \nabla f(x) \|_\infty \le U_{\sX} \epsilon. \\
    d^T X \nabla^2 f(x) X d = (T^{-1}d)^T \bar X \nabla^2 f(x) \bar X (T^{-1} d) \ge -\epsilon^p \| T^{-1}d \|^2 \\
    \ge -\epsilon^p \| x \|^2_\infty \| d \|^2 \ge -(U_{\sX}^{2/p} \epsilon)^p \| d \|^2.
\end{align*}
Then $x$ is also an $(\max\{U_{\sX}, U_{\sX}^{2/p}\}\epsilon,p)$-2o point by $\sodc$.
\item[(3)] Obviously we have that $\sodc \gtrsim \sode$. By (2) and the property of $\gtrsim$ and $\thickapprox$, $\sodd \gtrsim \sode$.
\item[(4)] Suppose that $x$ is an $(\epsilon,p)$-2o point by $\soda$. 
Let $J^+$, $J^-$, and $S$ be associated with $x$ as in $\soda$. 
Let $T = \diag(t)$ be a diagonal matrix of dimension $n \times n$ with  $t^i = 1$ for $i \in J^+$ and $t^i = x^i$, for $i \in J^-$. 
Then $X = T S$ and
\begin{align*}
 \| X \nabla f(x) \|_\infty \le \| X \nabla f(x) \| = \| T S \nabla f(x) \| \le \| T \| \| S \nabla f(x) \| \\
 \le \max\{ \| x \|_\infty, 1 \} \| S \nabla f(x) \|
 \le 2 \max\{ U_{\sX}, 1 \} \epsilon \le \max\{ U_{\sX}^{2/p},  2 U_{\sX} , 2 \} \epsilon.
\end{align*}
Also, for any $d \in \bR^n$, we have
\begin{align*}
d^T X \nabla^2 f(x) X d = d^T T S \nabla^2 f(x) S T d \ge - \epsilon^p \| T d \|^2 \\
\ge  - \epsilon^p \max\{ \| x \|_\infty^2  , 1 \}  \| d \|^2 = - (\epsilon \max\{ \| x \|_\infty^{2/p}  , 1 \}  )^p \| d \|^2 \\
\ge - ( \epsilon \max\{ U_{\sX}^{2/p}  , 1 \}  )^p \| d \|^2 \ge - (\epsilon \max\{ U_{\sX}^{2/p},  2 U_{\sX} , 2 \}  )^p \| d \|^2.
\end{align*}
If we let $c = \max\{ U_{\sX}^{2/p},  2 U_{\sX} , 2 \} $, then $x$ is an $(c\epsilon,p)$-2o point by $\sode$.
\ee
\qed
\end{proof}
The following example illustrates why {\sodb,\sodc,\sodd} are not essentially stronger than \soda.
\begin{example}\label{eg1}
Consider problem \eqref{opt: bc0},\eqref{def:O},\eqref{def:sI} in 1-dimension. Let $f(x) = \frac{1}{4} x^4$ and $p \in (0,1]$. Given any $c > 0$,  there exists $\bar \epsilon \in (0,1)$ such that for any $\epsilon \in (0, \bar \epsilon)$, we can find an $x \ge 0$ that is an $(\epsilon,p)$-2o point by \sodb, \sodc, \sodd, but not a $(c\epsilon,p)$-2o point by \soda. In particular, choose $\bar \epsilon$ such that for any $\epsilon \in (0, \bar \epsilon)$,
\begin{align}\label{ineq: eg.}
    \epsilon^{7/24} \le \Delta_{\max},\;  \epsilon^{1/6} \le \Delta_{\max}, \;
    \epsilon^{1/6} \le 6 \Delta_{\max}^2, \;
    \epsilon^{5/24} < c^{-1/2}, \;
    \epsilon^{1/8} < \frac{1}{2c}.
\end{align}
Let $x = \epsilon^{7/24}$. Then $f'(x) = x^3 = \epsilon^{7/8}, xf'(x) = x^4 = \epsilon^{7/6}, f''(x) = 3x^2 \ge 0, \forall x \ge 0$. Apparently, $x$ is an $(\epsilon,p)$-2o point by \sodc, \sodd. Note that by \eqref{ineq: eg.},
\begin{align*}
    & \min_{x + d \ge 0, |d| \le \Delta_{\max}} f'(x) d = -x^3 \cdot x = - \epsilon^{7/6} \ge - \Delta_{\max} \epsilon. \\
    & \min_{x + d \ge 0, |d| \le \Delta_{\max}} f'(x) d + \frac{1}{2} f''(x)d^2 = \min_{x + d \ge 0, |d| \le \Delta_{\max}} x^3 d + \frac{3x^2}{2} \cdot d^2 \\
    & \overset{ (d^* = -x/3)}{=}  -x^4/6 = -\epsilon^{7/6}/6 \ge - \Delta_{\max}^2 \epsilon^p. 
\end{align*}
Therefore, $x$ is also an $(\epsilon,p)$-2o point by \sodb (let $\Delta = \Delta_{\max}$). However, note that by \eqref{ineq: eg.},
\begin{align*}
    x > \sqrt{c\epsilon} \implies J^- = \{1\}; \; 
    f'(x) = \epsilon^{7/8} > 2 c \epsilon,
\end{align*}
so $x$ is not an $(c\epsilon,p)$-2o point by \soda.
\end{example}

A definition of approximate second-order optimality proposed in \cite{nouiehed2020trust}  is adapted to the scope of our study by setting $\Omega = \bR^n_+$, to obtain the following definition.
\begin{definition}[\cite{nouiehed2020trust}, $\sodf$] \label{def: nouiehed2020trust}
     $x$ is an $(\epsilon,p)$-2o point of \eqref{opt: bc0}, \eqref{def:O}, \eqref{def:sI} according to $\sodf$ if $x \in \Omega$, and
\begin{align}\label{eq: nouiehed2020trust}
\begin{aligned}
        \left| \mbox{globalmin}_{x + d \in \Omega, \| d \| \le  1} \quad  \nabla f(x)^T d  \right| &\le \epsilon, \\
    \left| \mbox{globalmin}_{x + d \in \Omega, \| d \| \le 1, \nabla f(x)^T d \le 0 } \quad d^T \nabla^2 f(x) d \right| &\le \epsilon^p.
\end{aligned}
\end{align}
\end{definition}
The next result states that $\sodb$ is essentially stronger than $\sodf$, which in turn is essentially stronger than $\sodc$.
\begin{theorem}\label{thm: nouiehed2020trust}
    (1). $\sodb \gtrsim \sodf$; (2). $\sodf \gtrsim \sodc$.
\end{theorem}
\begin{proof}
    \bi
    \item[(1)] Suppose that $x$ is an $(\epsilon,p)$-2o point by $\sodb$. Let $d_1^*$ and $d_2^*$ be the solution of the first and second problems in \eqref{eq: nouiehed2020trust}, respectively. WLOG, suppose that $\| d_1^* \| \neq 0$, $\| d_2^* \| \neq 0$. Define $ \bar {d_i} =  d_i^* \min\{ \Delta/\| d_i^* \|, 1 \}$, $i = 1,2$. Then $\bar d_1$ and $\bar d_2$ are feasible points in \eqref{eq:em2o-cartis}. Then we have
    \begin{align*}
        \nabla f(x)^T \bar d_1 & \ge -\Delta \epsilon, \\
        \implies  \nabla f(x)^T d_1^* & = \nabla f(x)^T \bar d_1 \max\{ \| d_1^*\|/\Delta,1\} \\
        & \ge -\Delta \max\{ \| d_1^*\|/\Delta,1\} \epsilon = - \max\{ \| d_1^* \|, \Delta \} \epsilon \ge - \max\{\Delta_{\max},1\} \epsilon.
    \end{align*}
Moreover, we have
\begin{align*}
        \frac{1}{2}\bar d_2^T \nabla^2 f(x) \bar d_2 + \nabla f(x)^T \bar d_2 \ge - \Delta^2 \epsilon^p,
    \end{align*}
    which implies that
    \begin{align*}
        \frac{1}{2}\bar d_2^T \nabla^2 f(x) \bar d_2 \ge - \Delta^2 \epsilon^p,
    \end{align*}
    since $\nabla f(x)^T \bar d_2 \le 0 $. Therefore, we have 
    \begin{align*}
        & \frac{1}{2} (d_2^*)^T \nabla^2 f(x) d_2^* = \frac{1}{2}\bar d_2^T \nabla^2 f(x) \bar d_2 (\max\{ \| d_2^* \|/\Delta,1 \})^2 \ge - \Delta^2 (\max\{ \| d_2^* \|/\Delta,1 \})^2 \epsilon^p \\
        & = - \max\{ \| d_2^* \|^2, \Delta^2 \} \epsilon^p \ge -\max\{ \Delta_{\max}^2, 1 \} \epsilon^p.
    \end{align*}
    Altogether, these expressions imply that $x$ is an $(c\epsilon,p)$-2o point by $\sodf$ where \\
    $c = \max\{ 2^{1/p}\Delta_{\max}^{2/p}, \Delta_{\max},2^{1/p} \}$.
    \item[(2)] Let $U_{\sX} \triangleq \max_{x \in \sX} \| x \|_\infty$. Suppose that $x$ is an $(\epsilon,p)$-2o point by $\sodf$. Similar to the proof of Theorem~\ref{thm: defcomp}, first row of \eqref{eq:em2o-Haeser} holds because $\nabla f(x) \ge -\epsilon {\bf 1}$, and $ \| X \nabla f(x) \|_\infty \le \max\{ U_{\mathcal{X}} , 1 \} \epsilon$. 
    Consider now the second row of \eqref{eq:em2o-Haeser}. If $x = 0$, this condition holds trivially. 
    When $x \neq 0$, let $d \triangleq c_d X v$, where $c_d \triangleq \min \left\{ \frac{1}{\| x \|_\infty}, 1 \right\}$, $v$ is an arbitrary vector with $\| v \|_2 = 1$. Let $\bar d = -{\rm sign}(\nabla f(x)^T d) d$. Therefore, we have that $x + \bar d \ge 0, \| \bar d \| \le 1$, $\nabla f(x)^T \bar d \le 0$. According to $\sodf$, we have
\begin{align*}
    \bar d^T \nabla^2 f(x) \bar d \ge -\epsilon^p
\end{align*}
Therefore,
\begin{align*}
    -\epsilon^p & \le \bar d^T \nabla^2 f(x) \bar d = d^T \nabla^2 f(x) d = c_d^2 v^T X \nabla^2 f(x) X v  \\
    \implies v^T X \nabla^2 f(x) X v & \ge - \epsilon^p/c_d^2 \ge - \max\{ \| x \|_\infty^2, 1 \} \epsilon^p \ge - \max\{ U_{\sX}^2, 1 \} \epsilon^p.
\end{align*}
Altogether, $x$ is an $(c\epsilon,p)$-2o point by $\sodc$, where
    $c = \max\{ U_{\mathcal{X}}^{2/p}, U_{\mathcal{X}},1 \}$.
    \ei
\end{proof}

\section{Capped conjugate gradient algorithm}\label{app: capped-CG}

The version of the conjugate gradient method shown in Algorithm~\ref{alg:ccg} was described in \cite[Algorithm~1]{Royer2019} to solve a system of the form $\bH y = -g$, where $\bH = H + 2\eps I$ is a damped version of the symmetric matrix $H$, which in our case is a principal submatrix of the  Hessian $\nabla^2 f(x_k)$. Note that the following results hold regarding Algorithm~\ref{alg:ccg}, as is shown in \cite[Lemma~3]{Royer2019} and \cite[Lemma~1]{Royer2019}.

\begin{lemma} \label{lm: capCG}
Consider the inputs $H,g,\epsilon$ and outputs d\_type and $d$ of Algorithm~\ref{alg:ccg}, if\\
1. d\_type=SOL, then
\begin{align*}
    d^T(H+2\epsilon I)d \ge \epsilon \| d \|^2,\quad 
    \| d \| \le 1.1 \epsilon^{-1} \| g \|,\quad 
    \| r \| \le \frac{1}{2}\epsilon \zeta \|d \|,
\end{align*}
where $r \triangleq (H + 2\epsilon I)d + g$.\\
2. d\_type = NC and  $\bar d \triangleq - \sgn (d^T g) \frac{d^T H d}{\|d \|^2} \frac{d}{\|d \|} $, then $\bar d^T g \le 0$ and
\begin{align*}
    \frac{\bar d^T H \bar d}{\| \bar d \|^2} = - \| \bar d \| \le -\epsilon.
\end{align*}

\end{lemma}
\rev{
\begin{lemma}\label{lm:CGoper.comp.}
Number of matrix-vector multiplication of Algorithm~\ref{alg:ccg} is bounded by
\begin{align*}
    2 \min\{n, \mathbb{J}(M,\epsilon,\zeta) \}+1,
\end{align*}
where
\begin{align*}
    \mathbb{J}(M,\epsilon,\zeta) \le \min\left\{n, \left\lceil \left( \sqrt{\kappa} + \frac{1}{2} \right) \log \left( \frac{144 (\sqrt{\kappa} + 1)^2 \kappa^6 }{\zeta^2} \right) \right\rceil \right\}.
\end{align*}
\end{lemma}}

\begin{algorithm}[ht!]
\caption{Capped Conjugate Gradient}
\label{alg:ccg}
\begin{algorithmic}
\STATE \emph{Inputs:} Symmetric matrix $H \in \bR^{n \times n}$; 
vector $g \ne 0$; damping parameter $\epsilon \in (0,1)$; desired relative 
accuracy $\zeta \in (0,1)$;
\STATE \emph{Optional input:} scalar $M$ (set to $0$ if not provided);
\STATE \emph{Outputs:} d\_type, $d$;
\STATE \emph{Secondary outputs:} final values of $M$, $\kappa$, $\hat{\zeta}$, $\tau$, and $T$;
\STATE Set
\[
\bH:=H+2\eps I, \quad \kappa := \frac{M+2\eps}{\eps}, \quad \hat{\zeta} := \frac{\zeta}{3 \kappa},
\quad
\tau := \frac{\sqrt{\kappa}}{\sqrt{\kappa}+1}, \quad
T:=\frac{4\kappa^4}{(1-\sqrt{\tau})^2};
\]
\STATE Set $y_0 \leftarrow 0$, $r_0 \leftarrow g$, $p_0 \leftarrow -g$, $j \leftarrow 0$; 
\IF{$p_0^\top \bH p_0 < \eps \|p_0\|^2$}
\STATE Set $d=p_0$ and terminate with d\_type=NC;
\ELSIF{$\|H p_0\| > M \|p_0\|$}
\STATE Set $M \leftarrow {\|H p_0\|}/{\|p_0\|}$ 
and update $\kappa,\hat{\zeta},\tau,T$ accordingly;
\ENDIF
\WHILE{TRUE}
\STATE $\alpha_j \leftarrow {r_j^\top r_j}/{p_j^\top \bH p_j}$;
\COMMENT{Begin Standard CG Operations}
\STATE $y_{j+1} \leftarrow y_j+\alpha_j p_j$;
\STATE $r_{j+1} \leftarrow r_j + \alpha_j \bH p_j$;
\STATE $\beta_{j+1} \leftarrow {(r_{j+1}^\top r_{j+1})}/{(r_j^\top r_j)}$;
\STATE $p_{j+1} \leftarrow  -r_{j+1} + \beta_{j+1}p_j$;
\COMMENT{End Standard CG Operations}
\STATE $j \leftarrow  j+1$;
\IF{$M < \max \left( \|Hp_j\|/\|p_j\|, \| Hy_j\|/\|y_j\|, \| Hr_j \|/\|r_j\| \right) $}
\STATE Set $M \leftarrow \max \left( \|Hp_j\|/\|p_j\|, \| Hy_j\|/\|y_j\|, \| Hr_j \|/\|r_j\| \right)  $
and update $\kappa,\hat{\zeta},\tau,T$ accordingly;
\ENDIF
\IF{$y_j^\top \bH y_j < \epsilon \|y_j\|^2$}
\STATE Set $d \leftarrow y_j$ and terminate with d\_type=NC;
\ELSIF{$\| r_j \| \le \hat{\zeta} \|r_0\|$}
\STATE Set $d \leftarrow y_j$ and terminate with d\_type=SOL;
\ELSIF{$p_j^\top \bH p_j < \eps \|p_j\|^2$}
\STATE Set $d \leftarrow p_j$ and terminate with d\_type=NC;
\ELSIF{$\|r_j\| >  \sqrt{T}\tau^{j/2} \|r_0\|$
}
\STATE Compute $\alpha_j,y_{j+1}$ as in the main loop above;
\STATE Find $i \in \{0,\dotsc,j-1\}$ such that
\begin{equation} \label{eq:weakcurvdir}
	\frac{(y_{j+1}-y_i)^\top \bH (y_{j+1}-y_i)}{\|y_{j+1}-y_i\|^2} \; < \; \eps;
\end{equation}
\STATE Set $d \leftarrow y_{j+1}-y_i$ and terminate with d\_type=NC;
\ENDIF
\ENDWHILE
\end{algorithmic}
\end{algorithm}

\section{Minimum eigenvalue oracle (MEO) procedure}\label{app: MEO}

The procedure shown as Procedure~\ref{alg:meo} is used to identify a direction of significant negative curvature, smaller  than a threshold $-\epsilon/2$ for a given $\epsilon>0$, or else return a certificate that all eigenvalues of $H$ are greater than $-\epsilon$. 
In the latter case, the certificate may be wrong, with probability up to a supplied tolerance $\delta$. 
This procedure is defined in \cite[Procedure~2]{Royer2019}, where a discussion of various possible implementations is given. 
The most interesting implementation for our purposes is a randomized Lanczos procedure, which performs a single  matrix-vector product involving $H$ at each of its iterations, and which finds the minimum eigenvalue of the projection of $H$ onto a Krylov subspace seeded by an initial random vector at each iteration.

\floatname{algorithm}{Procedure}

\begin{algorithm}[ht!]
\caption{Minimum Eigenvalue Oracle (MEO)}
\label{alg:meo}
\begin{algorithmic}
  \STATE \emph{Inputs:} Symmetric matrix $H \in \bR^{n \times n}$, tolerance $\epsilon>0$, error probability $\delta \in [0,1)$;
  \STATE \emph{Optional input:} Upper bound on Hessian norm $M$;
  \STATE \emph{Outputs:} An estimate $\lambda$ of $\lambda_{\min}(H)$ such that $\lambda \le - \epsilon/2$, and vector $v$ with $\|v\|=1$ such that $v^\top H v =\lambda$ {\bf OR} a certificate that $\lambda_{\min}(H) \ge -\epsilon$. 
  The probability that the certificate is issued but $\lambda_{\min}(H)<-\epsilon$ is at most $\delta$.
\end{algorithmic}
\end{algorithm}

Based on the discussion in \cite[Section~3.2]{Royer2019} and \cite[Assumption 3]{Royer2019}, we have the following result about bound on Hessian-vector products when a randomized Lanczos procedure (or a randomized CG) is used to implement Procedure~\ref{alg:meo}. 
\rev{
\begin{lemma}\label{lm:MEOoper.comp.}
We use a randomized Lanczos method with a starting vector uniformly generated on a unit sphere to implement Procedure~\ref{alg:meo}. Then given a failure probability $0 < \delta \ll 1$, Procedure~\ref{alg:meo} either certifies that $H \succeq -\epsilon I$ or finds a direction along which curvature of $H$ is smaller than $-\epsilon/2$ in at most $\Nmeo \triangleq \min \{ n, 1 + \lceil \Cmeo \epsilon^{-1/2} \rceil \}$
Hessian-vector products, where $\Cmeo = \log (2.75n/\delta^2)\sqrt{ \| H \| \max\{1,\epsilon^2\} }/2.$ \revn{If $\lambda_{\min}(H) < -\epsilon$, then with probability at most $\delta$, a certificate will be given.}
\end{lemma}}

\section{Two-sided bounds} \label{app: 2sidebds}
In this section we consider the two-sided bound-constrained optimization:
\begin{equation} \label{opt: gbc}
\min \, f(x)  \quad \mbox{s.t.} \;\;
x \in \Omega \triangleq \{x \in \bR^n \mid 0 \le x^i \le u^i, \; i \in \sI \}
\end{equation}
where $f: \bR^{n} \rightarrow \bR$ is twice continuously
differentiable and is bounded below by $\fl$ on the feasible region $\Omega$,
and $\sI \subset \{1,2,\dotsc,n\}$. We assume without loss of generality that $u^i > 0$ for all $ i \in \sI$. We allow $u_i = \infty$, that is, not all components $x_i$ for $i \in \sI$ have upper bounds.

Extending Definition~\ref{def:em2o}, we define approximate optimality for \eqref{opt: gbc} as follows.
\begin{definition}[$(\epsilon,p)$-2o of \eqref{opt: gbc}] We say that $x$ is an $(\epsilon,p)$-2o point for  \eqref{opt: gbc} if and only if
\begin{align*}
& 0 \le x^i \le u^i, \; i \in \sI, \; \| S \nabla f(x) \| \le 2\epsilon, \mbox{ and }
\begin{cases}
\nabla_i f(x) \ge -\epsilon^{3/4}, &  \; i \in \sI, x_k^i \le \sqrt{\epsilon}, \\
\nabla_i f(x) \le \epsilon^{3/4}, & \; i \in \sI, x_k^i \ge u^i - \sqrt{\epsilon},
\end{cases} \\
& S \nabla^2 f(x) S \succeq -
  \epsilon^p I,
\end{align*}
\end{definition}
where we define $J^+ \triangleq \{ i \in \sI \mid 0 \le x^i \le \sqrt{\epsilon} \mbox{ or } u^i - \sqrt{\epsilon} \le x^i \le u^i \}$, $J^- \triangleq \{ 1,\hdots,n \} \setminus J^+$,
$S = \diag({s})$, where
$s^i = \min\{  x^i, u^i - x^i \}$ if $i \in J^+$, and $s^i= 1$ if $i \in J^-$. 
Again, this definition reduces to Definition~\ref{def:em2o} when $u^i = +\infty$ for all $i \in \sI$, and can be motivated by exact (weak) second-order optimal conditions of \eqref{opt: gbc}.
The extension of projected Newton-CG (Algorithm~\ref{Alg: PNCG}) to the general bound-constrained optimization \eqref{opt: gbc} is relatively straightforward. 
We redefine the projection operator $P$, index sets $J_k^+$ and  $J_k^-$, and $S_k=\diag(s_k)$ as follows:
\begin{align*}
[P(x)]^i & \triangleq \begin{cases}
\mbox{mid}  (0,x^i,u^i) & \; i \in \sI, \\
x^i & \; \mbox{otherwise},
\end{cases} \\
J_k^+ & \triangleq \{ i \in \sI \mid 0 \le x_k^i \le \epsilon_k \mbox{ or } u^i - \epsilon_k \le x_k^i \le u^i \}, \\
J_k^- & \triangleq \{ 1,\hdots,n \} \setminus J_k^+ = \{ i \in \sI \mid \epsilon_k < x_k^i < u^i - \epsilon_k \} \cup \sI^c. \\
s_k^i & = \begin{cases} \min\{ x_k^i, u^i - x_k^i \}, & \mbox{if } i \in J_k^+, \\
1, & \mbox{otherwise.}
\end{cases}
\end{align*}
The definitions of $g^-_k$, $H^-_k$, $g^+_k$ and $S^+_k$,  are the same, modulo the redefined $P$, $J_k^+$, $J_k^-$, and $S_k$. 
For Algorithm~\ref{Alg: PNCG}, the only adjustment to be made is the conditions to trigger the gradient step, which become
\[
g_k^i < -\epsilon_k^{3/2}, x_k^i \le \epsilon_k, i \in \sI \mbox{ or } g_k^i > \epsilon_k^{3/2}, x_k^i \ge u^i - \epsilon_k, i \in \sI \mbox{ or } \| S^+_k g^+_k \| \ge \epsilon_k^2. \]
We make an additional assumption on $\epsilon_k$ that 
\[
2\epsilon_k \le u^i, \quad \forall k \ge 0, \ i \in \sI,
\]
and  assume that Assumption~\ref{Ass: comp.lev.set} and \ref{Ass: 2Lipstz} hold when $\Omega$ includes two-sided bounds.
It can then be verified that Lemmas~\ref{lm: ncdec},~\ref{lm: soldec} and~\ref{lm: meostep} still hold for the modified Algorithm~\ref{Alg: PNCG}. 
Lemma~\ref{lm: gradproj} also holds if the conditions to trigger the gradient step is tailored accordingly. Furthermore, if we let $\epsilon_k \equiv \epsilon_H = \sqrt{\epsilon}$ and $\epsilon_g = \epsilon$, then the Algorithm stops within the same number of iterations specified in Theorem~\ref{thm: PNCG.iter.comp} ($\sO(\epsilon^{-3/2})$) and locates an $x$ that is an $(\epsilon,1/2)$-2o point of \eqref{opt: gbc} with probability at least $(1-\delta)^{\Kpncg}$, where $\delta \in [0,1)$ is the probability of failure in Procedure~\ref{alg:meo}. Moreover, the complexity of fundamental operations (gradient evaluations or Hessian-vector products) is also $\tilde{O}(\epsilon^{-7/4})$.

\end{document}